%% file: paper.tex
\documentclass[12pt,a4paper]{article}
\usepackage{graphicx} 
\usepackage[utf8]{inputenc}

\usepackage{xcolor}
\usepackage{amsmath,amsthm}
\usepackage{amssymb}
\usepackage{comment}
\usepackage{hyperref}
\usepackage{cleveref}
\usepackage{enumitem}
\usepackage{bm}
\usepackage{tikz,pgfplots,float}

\usepackage{geometry}
\theoremstyle{definition}
\newtheorem{theorem}{Theorem}[section]

\newtheorem{lemma}[theorem]{Lemma}

\newtheorem{proposition}[theorem]{Proposition}

\newtheorem{assumptions}[theorem]{Assumptions}
\newtheorem{corollary}[theorem]{Corollary}

\newtheorem{remark}[theorem]{Remark}
\newtheorem{example}[theorem]{Example}

\newcommand{\limitlambda}{\lim_{\lambda>0,\, \lambda\rightarrow 0}}

\newcommand{\R}{\mathbb R}

\newcommand{\N}{\mathbb N}

\newcommand{\T}{\mathcal T}
\newcommand{\op}{\mathrm{op}}
 
\newcommand{\EE}{\mathbb E} 
\newcommand{\E}{\EE} 
\newcommand{\mean}{\EE} 
\newcommand{\Var}{\mathrm{Var}} 

\newcommand{\cond}[1]{\, #1 \vert \,} 
\newcommand{\norm}[1]{\left\lVert#1\right\rVert}
\newcommand{\abs}[1]{\left |#1\right|}

\newcommand{\<}{\langle}
\renewcommand{\>}{\rangle}

\newcommand{\Ac}{\mathcal A}
\newcommand{\Pbb}{\mathbb P}

\newcommand{\Lc}{\mathcal L}

\title{Monte Carlo on a single sample}
\author{Nils Detering, Nicole Hufnagel, Paul Kr\"uhner}
\date{\today}

\begin{document}
\maketitle
\begin{abstract}
    In this paper, we consider a Monte Carlo simulation method (MinMC) that approximates prices and risk measures for a range $\Gamma$ of model parameters at once. The simulation method that we study has recently gained popularity \cite{HugeSavine2020,RolfP-2025,BenthDeteringGalimbertiFlowForwards}, and we provide a theoretical framework and convergence rates for it. In particular, we show that sample-based approximations to $\E_{\theta}[X]$, where $\theta$ denotes the model and $\E_{\theta}$ the expectation with respect to the distribution $P_\theta$ of the model $\theta$, can be obtained across all $\theta \in \Gamma$ by minimizing a map $V:H\rightarrow \mathbb{R}$ with $H$ a suitable function space. Minimization can be achieved easily by fitting a standard feedforward neural network with stochastic gradient descent. We show that MinMC, which uses only one sample for each model, significantly outperforms a traditional Monte Carlo method performed for multiple values of $\theta$, which are subsequently interpolated. Our case study suggests that MinMC might serve as a new benchmark for parameter-dependent Monte Carlo simulations, which appear not only in quantitative finance but also in many other areas of scientific computing.
\end{abstract}

\section{Introduction}

We consider a set $\Gamma$ of risk neutral models for pricing options, where $(\Gamma,d_\Gamma)$ is a separable, metric space and $\mathcal B$ its Borel-$\sigma$-algebra. Under a fixed model $\theta\in\Gamma$, we denote the price process defined on some measurable space $(\Omega,\Ac)$ by $(S_t^\theta)_{t\geq0}$ and choose a payoff function $p$. Hence, the price of the European option with payoff function $p$ in the model $\theta$ is
\begin{align}
    \label{eq:def_C}
    C(\theta):=\mean [ p (S_T^{\theta}) ].
\end{align}
When no analytic expression for the price $C(\theta)$ is available in the model $\theta$, usually a Monte Carlo simulation is performed by simulating $M$ i.i.d.\ samples of $p(S_T^\theta)$ and calculating 
\begin{align}
\label{eq:MC_for_C}
    C(\theta)=\lim\limits_{M\to \infty} \frac{1}{M} \sum_{i=1}^M p\big(S_T^\theta(\omega_i)\big) .
\end{align}
This methodology is widely adopted (see for instance \cite{glasserman2004}), and according to the Central Limit Theorem (CLT), cf. \cite{pages2018}, the variance associated with the approximation error is $\text{Var} (p(S_T^\theta))/ M$. In this context, we fix $\theta$ and undertake a Monte Carlo simulation for the selected model. When considering an alternative model, such as a slight modification of a parameter, $M$ additional samples are required. For minor variations in $\theta$, it is anticipated that the price will not change significantly. However, with the new full Monte Carlo simulation, the information derived from our prior samples remains completely unused. Our objective is to implement this approach more efficiently with a procedure that uses as little as one sample for each model $\theta$ but makes use of the continuity of the map $\theta \mapsto C(\theta)$. We derive convergence rates for the approximation error for this procedure across the entire set of models $\Gamma$. We emphasize that the problem of calculating $C(\theta) $ for all $\theta \in \Gamma$ is particularly important for applications. The parameters of models used at a derivative trading desk are frequently changing, often necessitating Monte Carlo simulations to be rerun whenever there is a change in the price of the underlying assets. In addition, for calibration of a model to the market or the calculation of Greeks, one needs option prices for a range of model specifications. Therefore, a procedure capable of directly calculating the entire function $C : \Gamma \to \mathbb{R}$ carries a substantial advantage. Moreover, our more general setup introduced below is suitable for most parameter-dependent Monte Carlo simulations. It therefore applies not only to option pricing, but also to many risk management tasks in the banking and insurance business where it is often important to calculate expectations for a range of different model parameters to access model risk (see, for example, \cite{ContModelRisk,Detering2016,JokhadzeSchmidt2020,Crepey2024,LazarQiTunaru}).

For the approach to Monte Carlo simulation considered in this paper, we construct an extension $X:\Omega\times \Gamma \to \R$ for the payoff, where, for a given model $\theta\in \Gamma$, the distribution of $X(\cdot,\theta)$ is known (for example $X(\omega,\theta):=p(S_T^\theta(\omega))$ for the case of a European payoff mentioned above). Then, we specify a random variable $\Theta$ with distribution $\mu$ that characterizes the model. Let \( h: \Gamma \rightarrow \mathbb{R} \) be a measurable function such that all expectations below are finite. Given the particular specification of \( \Theta \) and the probability measure \( \mathbb{P} \) on \( (\Gamma \times \Omega, \mathcal{B} \otimes \mathcal{A}) \) described in Section~\ref{sec:mainresults}, the following calculation holds:

\begin{align}
\begin{split}
\label{eq:argmin_V}
    & V(h) := \mathbb{E} [ |X - h(\Theta)|^2 ] = \int_\Gamma\mathbb{E_{\theta}} [  \abs{X - h(\theta ) }^2   ]d \mu (\theta) \\
    &=\int_\Gamma \mathbb{E}_\theta [ X^2 ] -\mean_\theta[X]^2 d \mu (\theta) +  \int_\Gamma \big(\mean_\theta[X]^2-2\mathbb{E}_\theta [ X ] h(\theta ) +h(\theta )^2 \big) d \mu(\theta) \\
    &\stackrel{\eqref{eq:def_C}}{=}  \int_\Gamma  \Var_\theta [ X] d \mu  +  \int_\Gamma \big( \mathbb{E}_\theta [ X ] - h(\theta )  \big)^2  d \mu(\theta),
    \end{split}
\end{align}
where $\text{Var}_\theta$ and $\mathbb{E}_\theta$ are such that  $\E[X\cond{}\Theta] = (\E_{\theta}[X])|_{\theta=\Theta}$ and $\text{Var}[X\cond{} \Theta] =(\text{Var}_\theta [X])_{\theta=\Theta} $.
We notice that only the second term in \eqref{eq:argmin_V} actually depends on $h$. Thus, we can minimize the above expression $h\mapsto V(h)$ by choosing $h_0 (\theta ) = \mathbb{E}_\theta [ X ]$, which is the price of the option in the model $\theta$. This is basically a random field extension (see, for instance \cite{HugeSavine2020,Beck_2021}) of the well-known $\mathcal L^2$ approximation of expectation.

In summary, the aim is to find a function $h_0:\Gamma \to \R $ satisfying  
\begin{align*}
    h_0(\Theta)= \mean[X\cond{}\Theta]
\end{align*}
using a numerical method when the laws of $\Theta$ and the conditional law $P_\theta^X$ of $X$ given $\Theta=\theta$ are known but there are no analytic formulas available for calculating the conditional expectation. As in \eqref{eq:MC_for_C}, this could be achieved numerically via choosing $N$ positions $\Theta_1,\dots,\Theta_N$ and approximating $h(\Theta_i)$ with $M$ Monte Carlo simulations each. Based on the previous considerations, however, we proceed differently. We begin by approximating the function \( h \mapsto V(h) \) based on $N$ i.i.d.\ observations of $(X,\Theta)$. Afterwards, we compute the minimum $h_N$ of this approximation and finally prove that $h_N$ converges to the minimum $h_0$ of \( V \) for $N\rightarrow \infty$, provided a ridge term is added that prevents overfitting. The ridge term needs to converge to $0$, but at a rate that is not too fast. In summary, we choose a sample-dependent minimization and show that it converges to the minimum of $V$. The minimization step can, for example, be performed using a stochastic gradient descent algorithm that determines optimal weights for a neural network that approximates $h_N$. In the following, we often refer to this method as MinMC. For each fixed model $\theta$, in contrast to a classical Monte Carlo simulation, with the baseline MinMC at most one sample is actually generated ($M=1$). This is possible due to an explicit use of the continuity properties of the function $h_0:\Gamma \rightarrow \mathbb{R}$. Although our results cover the extreme case $M=1$, we show that it is often advisable to choose $M$ a bit larger, where in our case study, the optimal choice for $M$ ranges between $1$ and $100$ depending on the problem. Within our case study, we demonstrate that determining $\mean[X\cond{}\Theta]$ as the minimizer of a sample-based approximation of the mapping $h\mapsto V(h)$ significantly surpasses the performance of traditional Monte Carlo simulations conducted across multiple values of $\theta$, which are subsequently interpolated. Although our case studies focus primarily on option pricing, the methodology is general and applicable across various domains in scientific computing that involve Monte Carlo simulations which depend continuously on a parameter. 

To prove our convergence results, we understand the function $u_0$ as an element in a reproducing kernel Hilbert space $H$. This allows us to rewrite the sample-based approximation of $V$ plus a ridge term as a convex quadratic operator whose minimizer can be shown to be unique and explicitly expressed. Resorting to a series of operator-theoretic results allows us to obtain our convergence rates.

\subsection{Literature review}
The methodology for computing $\mean[X\cond{}\Theta]$ in the context of derivative pricing via the minimization of a functional analogous to $h\mapsto V(h)$ in a class of neural networks was initially proposed in \cite{HugeSavine2020} and compared with more classical regression techniques in \cite{RolfP-2025}. In \cite{HugeSavine2020} a heuristic convergence argument is provided, but without convergence rates. In \cite{Beck_2021} a similar method is used for solving the Kolmogorov partial differential equation using deep learning based on the Feynman-Kac formula. Notably, \cite{Beck_2021} demonstrates that the minimizer coincides with the solution of the targeted PDE, which is subsequently approximated using neural networks. In \cite{BenthDeteringGalimbertiFlowForwards} continuity results for the maps $h\mapsto V(h)$ and $\theta\mapsto h(\theta)$ have been established to adapt the method for flow forward pricing in commodity markets where the situation is further complicated by the infinite-dimensional nature of $\Gamma$. However, none of these papers addresses the question of an approximation with finitely many samples and its related convergence rates. This paper closes this gap in the literature by conducting a rigorous analysis and deriving convergence rates of MinMC within a broad framework, highlighting its applicability across diverse scenarios but with a particular focus on option pricing. Our case study suggests that MinMC might serve as a new benchmark for parameter-dependent Monte Carlo simulations that appear not only in quantitative finance but in many areas of scientific computing.

From a technical perspective, our work has some connection to the field of functional regression where the goal is to infer an unknown function \( f: \mathbb{R}^d \to \mathbb{R} \). In this context, \( f(\theta) \) is approximated based on samples \(\{(\Theta_i, X_i)\}_{i=1}^n\), where
\( \Theta_i \) represents an input and \( X_i \) is the observed response generated by the error model 
\begin{equation}\label{error:model}
 X_i = f(\Theta_i) + \varepsilon_i   ,
\end{equation}
    where \( \varepsilon_i \) is a noise term with \( \E[\varepsilon_i] = 0 \) and \( \text{Var}(\varepsilon_i) = \sigma^2 \).
The most common way for estimating $f:\Gamma \rightarrow \mathbb{R}$ is by means of {\em Parametric regression} (see, for example \cite{Draper1998,Montgomery2012}), where one assumes a fixed functional form (e.g., linear for linear regression) which has the advantage that the number of parameters to be estimated is usually low. However, it comes at the cost of a restricted shape for $f$. An alternative is {\em nonparametric estimation} (e.g., kernel regression, splines) where no specific structure is assumed (see \cite{Stone1977,Stone1980,Hardle1985} and the comprehensive manuscript \cite{Hardle1990}) and therefore more general functions $f$ can be estimated. Non-parametric regression techniques can be used for estimating conditional expectation when $f$ is given by $f(\Theta)= \mean[X\cond{}\Theta]$. More closely related to our approach is reproducing kernel Hilbert space (RKHS) regression. Kernel Hilbert space regression is regarded as a non-parametric regression approach wherein the selected RKHS imposes constraints on the regression function. This has recently been used for estimating the yield curve \cite{filipovic2022stripping}, for statistical estimation of conditional expectation and variance from time-series in \cite{filipovic2025jointestimationconditionalmean} and for uniform function estimation in \cite{dommel2021uniformfunctionestimatorsreproducing}. Qualitatively, there is a key difference between the regression problem and MinMC. In fact, for MinMC, the conditional distribution of $X$ given $\Theta_i$, is known and one can sample from it. The goal is then an efficient numerical technique to approximate $h_0(\Theta)= \mean[X\cond{}\Theta]$. The reproducing kernel Hilbert space approach comes in handy for MinMC as it imposes only minimal structure on the function $h_0$ and its approximating functions, and therefore in particular allows for approximations with neural networks and the stochastic gradient descent optimizer.

Our paper is organized as follows. In \Cref{sec:mainresults}, we summarize our main mathematical result \Cref{t:approximation} and provide some specific frameworks with examples. In \Cref{sec:proofs}, we provide a detailed proof of the main result as well as several helpful lemmas. \Cref{sec:case_study} contains two case studies. The first highlights our method on the Black-Scholes model where simple benchmarking is possible. The second is applied to the Heston model where no closed-form price formula is known. Additional technical results can be found in the appendix.

\section{Main results}
\label{sec:mainresults}
We consider a random variable $X:\Omega\rightarrow \mathbb R$, which specifies the payoff. Let $P$ be a probability transition kernel from $\Gamma$ to $\Omega$, that is, $P:\Gamma\times \mathcal A\rightarrow [0,1]$ with 
\begin{itemize}
    \item[(i)] $P_\theta$ is a probability measure on $(\Omega,\mathcal A)$ for any $\theta\in \Gamma$ and 
    \item[(ii)] $\theta\mapsto P_\theta(A)$ is measurable for any $A\in\mathcal A$. 
\end{itemize}

For a probability measure $\mu$ on $(\Gamma,\mathcal B)$ we define the probability measure $\Pbb$ on $(\Gamma\times \Omega, \mathcal B\otimes \mathcal A)$ via 
 $$ \Pbb(B\times A) := \int_B P_\theta(A) \mu(d\theta),\quad B\in\mathcal B,A\in\mathcal A.$$

We introduce on $(\Gamma\times \Omega, \mathcal B\otimes \mathcal A)$ the random variable $\Theta(\theta,\omega):=\theta$, with $\theta\in\Gamma$ and $\omega\in\Omega$, to describe the distribution of the possible models. We observe that its law under $\Pbb$ is given by $\Pbb^\Theta=\mu$.
Extending the random variable $X$ to the product space $\Gamma\times \Omega$ by $X(\theta,\omega):=X(\omega)$, we find that
 $$ \E[X\cond{}\Theta] = (\E_{\theta}[X])|_{\theta=\Theta} $$
if $X\geq 0$ or $\E[|X|]<\infty$ where $\E_{\theta}$ denotes expectation under $P_\theta$ and $\mean$ expectation under $\Pbb$. We denote the corresponding $\Lc^p$-spaces by
\begin{align*}
    &\Lc^p(\mu):=\Big\{ g:\Gamma\to \R \cond{}g\text{ measurable, } \|g\|_{\Lc^p(\mu)}:=\Big(\int_\Gamma |g(\theta)|^pd\mu(\theta)\Big)^{\frac1p}<\infty \Big\}, \\
    &\Lc^p(\Pbb,A):=\Big\{g:\Gamma\times \Omega\to A\cond{} g \text{ measurable, } \|g\|_{\Lc^p(\Pbb,A)}:=\big(\mean(|g|^p)\big)^{\frac1p}\Big\}.
\end{align*} 
For simplicity we omit the image space $A$ in the case of real functions. In case $X\in \Lc^2(\mathbb P)$, then by construction, it holds that $\text{Var}[X\cond{} \Theta] =(\text{Var}_\theta [X])_{\theta=\Theta} $ where $\text{Var}_\theta$ denotes the variance under $P_\theta$. In this setting, the motivating calculations presented in \eqref{eq:argmin_V} apply. Based on these considerations, our goal is to determine the argmin of $V$ over a suitable function space. 

First, we specify the function space. 
Let $H\subset C(\Gamma,\mathbb R)$ 
be a Hilbert space such that the point evaluation in $H$ is a continuous linear functional. We denote the skalar product of $H$ by $\langle \cdot, \cdot\rangle$ and its norm by $\|\cdot \|$. Due to Riesz representation theorem \cite[Thm. 12.5]{rudin1991functional}, there exists a function $k : \Gamma \times \Gamma \to \mathbb{R}$ such that $k_t:=k(t,\cdot)\in H$ and $\<h,k_t\>=h(t)$ for $h\in H$, $t\in \Gamma$. In other words, $H$ is a reproducing kernel Hilbert space, see \cite{manton2015rkhs,Fukumizu2011rkhs}. 

In this section, we present our main results and provide examples for the space $H$. Consider the i.i.d.\ sequence $\{(X_n,\Theta_n)\}_{n\in\N}$ where $(X_1,\Theta_1)$ has the same law as $(X,\Theta).$ We proceed under the following assumptions: 
\begin{assumptions}
\ \vspace{-4mm}
\begin{enumerate}[label=(A\arabic*), leftmargin=2cm]
        \item \label{ass:k:finite:secondmoment} The random variable $\theta\mapsto k(\theta,\theta)$ has finite second $\mu$-moment: $$\int_{\Gamma} \big(k(\theta,\theta)\big)^2 \mu(d\theta)<\infty.$$
        \item \label{ass:support:mu}The support of $\mu$ is $\Gamma$:
       $$ \Gamma = \bigcap\{ A\subseteq \Gamma: A\text{ closed, }\mu(A)=1\}.$$
       \item \label{ass:finite:fourth:moment}$X:\Gamma \times\Omega\rightarrow \mathbb R$ is a random variable which satisfies $\E[ X^4]<\infty$.
       \item \label{ass:minimizer:exists}There is $h_0\in H$ such that 
   $$ \E[X\cond{}\Theta] = h_0(\Theta). $$
    \end{enumerate}
\end{assumptions} 
\ref{ass:k:finite:secondmoment} guarantees $\E[h(\Theta)^4]<\infty$ for any function $h\in H$, see \Cref{l:norm of h(theta)}. The support property \ref{ass:support:mu} of $\Theta$ yields that any two elements of $H$ which are $\Pbb^\Theta$-a.s. equal are equal, see \Cref{l:as is equal}.

We consider the quadratic functions $V:H\rightarrow \mathbb R$ and $V_N:H\rightarrow \Lc^2(\Pbb)$ given by
\begin{align*}
    V(h) &\hspace{1mm}=\E[ (h(\Theta)-X)^2 ] \\
    V_N(h) &:=  \frac1N\sum_{i=1}^N (h(\Theta_i)-X_i)^2
\end{align*}
for any $h\in H$. In addition, we introduce regularised versions that include a ridge term; for $\lambda>0$ we define
 $$ V_\lambda(h) := V(h) + \lambda \|h\|^2,\quad V_{N,\lambda}(h):=V_N + \lambda  \|h\|^2. $$ 
 We note that the minimizer of $V$ is the function $h_0$, see \ref{ass:minimizer:exists}. According to our goal, we need to approximate $h_0$ from a sample $(X_1,\Theta_1),\dots,(X_N,\Theta_N)$. 
     For a given sample $(X_1,\Theta_1),\dots,(X_N,\Theta_N)$ the penalty $V_{N,\lambda}$ is explicitly specified via
     $$ V_{N,\lambda}(h) = \frac1N\sum_{i=1}^N (h(\Theta_i)-X_i)^2 + \lambda \|h\|^2. $$
    Optimizing this penalty can be achieved by large enough neural network, cf.\ universal approximation theorem in \cite{hornik1989}. The obtained network $g$ satisfies $g\approx h_{N,\lambda}$ and with the choice of $\lambda$ in Theorem \ref{t:approximation} below, we can then ensure that $g\approx h_0$.
 
 The proof of the following theorem, which establishes a bound on the convergence rate of MinMC, is given in \Cref{sec:proofs}.

\begin{theorem}\label{t:approximation}
    The minimizer $h_{N,\lambda}$ of $V_{N,\lambda}$ exists for any $N\in\N, \lambda >0$. There is a constant $C> 0$ such that
    \begin{equation}\label{main:conv:rate}
    \E[ \|h_{N,\lambda_N}-h_0\|_{\mathcal L^2(\mu)}] \leq C\max\big\{\lambda_N^{-2}N^{-\frac12},\sqrt{\lambda_N}\big\}, \quad N\in\mathbb N 
    \end{equation}
    for any $N\in\mathbb N$, $\lambda >0$.
\end{theorem}
\begin{remark}
\label{rem:convergencerate}
    For the application of \Cref{t:approximation} we assume that $\lambda_N\to 0$ and $\lambda_NN^{\frac14}\to \infty$ to ensure $\lim\limits_{n\to\infty}\E[ \|h_{N,\lambda_N}-h_0\|_{\mathcal L^2(\mu)}]=0$. Since $\lambda_N^{-2}N^{-\frac12}$ decreases monotonically and $\sqrt{\lambda_N}$ increases monotonically in $\lambda_N$, the optimal rate in \eqref{main:conv:rate} is obtained when $ \lambda_N^{-2}N^{-\frac12}=\sqrt{\lambda_N}$. Hence, the optimal rate for the right hand side is $N^{-\frac1{10}} $ when $\lambda_N=N^{-\frac15}$. Compared to the classical Monte Carlo rate of $N^{-1/2}$, this rate may not look impressive. However, it should be noted that the approximation is applied throughout the model class $\Gamma$ instead of only for one model $\theta \in \Gamma$. Moreover, the rate is independent of the dimension of $\Gamma$. Our case study in Section \ref{sec:case_study} shows a superior performance of MinMC even when compared on a finite deterministic subset $\theta_1, \dots , \theta_K$ with classical Monte Carlo simulations performed exactly for the models $\theta_1, \dots , \theta_K$. Like all Monte Carlo methods, a larger sample yields a better theoretical approximation. In addition, a larger network and longer training should give a better approximation $g$ of the approximator $h_{N,\lambda}$. 
\end{remark}

For a neural network $g$ with $g\approx h_{N,\lambda}$, it is important to also understand the mse-loss of the network. The next proposition shows that the mse-loss is approximately equal to the expected conditional variance of $X$. This is despite the fact that we already know that $h_{N,\lambda}\rightarrow h_0$ for $N\rightarrow \infty$ and $\lambda \rightarrow 0$.

 The proof of the following proposition is given in \Cref{sec:proofs}.
\begin{proposition}\label{p:loss}
    The minimizer $h_{N,\lambda}$ of $V_{N,\lambda}$ exists for any $N\in\mathbb N$, $\lambda>0$. For $\lambda_N\to 0$ and $\lambda_NN^{\frac14}\to \infty$ the following convergence holds:
  $$ \lim\limits_{N\to \infty}\frac{1}{N}\sum_{i=1}^N (h_{N,\lambda_N}(\Theta_i)-X_i)^2 = \E[\Var(X\cond{}\Theta)]$$
  as $N\rightarrow \infty$ in $\Lc^1(\Pbb)$.
\end{proposition}

In the following, we provide examples that align with our framework. In each case, we begin by introducing our set of models $\Gamma$, a subset of $\R^n$, along with the corresponding Hilbert space $H$ and the function $k$ that satisfies all the required properties. Then, we present an example for the payoff $X$ that satisfies \ref{ass:finite:fourth:moment} and \ref{ass:minimizer:exists}. Our last Example \ref{exm:featuremap}, which uses feature maps to map the problem into a simple convenient $\mathcal{L}^2$ space, provides a very general framework that also turns out very efficient from a computational perspective, as we will see in our case study Section \ref{sec:case_study}. 

First, we focus on translation invariant $k$, that is, there exists a function $\psi$ such that $k(x,y)=\psi(x-y)$, see \cite[p. 2393]{Fukumizu2011rkhs}. 
\begin{lemma}
\label{lem:translationinvariant_kernel}
    Let $\Gamma \subset \R^n$. If there exists a  continuous, positive-definite and symmetric function $\psi:\R^n\to \R$ such that $k(x,y)=\psi(x-y)$ then $k$ 
    satisfies \ref{ass:k:finite:secondmoment} for any measure $\mu$ on $\Gamma$. Furthermore, there is a Hilbert space
     $$ H\subseteq C(\Gamma,\mathbb R)$$
     with some scalar product such that 
     \begin{itemize}
         \item[(i)] $k_x\in H$ for any $x\in\Gamma$ and
         \item[(ii)] $\<f,k_x\>=f(x)$ for any $x\in H$.
     \end{itemize}
\end{lemma}
\begin{proof}
    We have $k(\theta,\theta)=\psi(0)$ which guarantees \ref{ass:k:finite:secondmoment} for any measure $\mu$ on $\Gamma$, that is, finite second $\mu$-moment. The function $k$ receives the properties continuous, positive-definite and symmetric from $\psi$. Hence, the existence of a Hilbert space satisfying (i) and (ii) follows by the Moore-Aronszajn Theorem \cite[p. 344]{aronszajn1950theory}. 
\end{proof}

\begin{example}
\label{exm:[0,1]}
    We consider the Hilbert space 
    \begin{align*}
        H&:=\big\{f:[0,1]\rightarrow \mathbb R\cond{} f\text{ absolutely continuous with }f'\in \mathcal L^2(\mathcal U([0,1])\big\}\\
        &\hspace{1mm}
        \subseteq C([0,1],\R)
    \end{align*}
    with scalar product
     $$ \<f,g\> := \frac12\big(f(1)+f(0)\big)\big(g(1)+g(0)\big) + \frac12\int_0^1 f'(x)g'(x)dx. $$
     Here, $\Gamma:=[0,1]$ specifies the models for pricing options with a continuous uniform distribution $\mu:=\mathcal U([0,1])$. For $x\in[0,1]$ the function $k_x: [0,1] \to \R$ given via $k_x(y):= 1 - |x-y|$ fulfills $k_x\in H$ and
      $$ \<k_x,f\> = f(x)$$
      for $f\in H$. Consequently,
     $$ k(x,y) := k_x(y) = 1 - |x-y|$$
    satisfies the desired properties. 
    
    Assumption \ref{ass:k:finite:secondmoment} holds by \Cref{lem:translationinvariant_kernel} since $k$ is a translation-invariant function. 
    The support property \ref{ass:support:mu} is valid due to the choice of $\mu$. Apparently, we can choose a different $\mu$ such that \ref{ass:support:mu} remains fulfilled. Furthermore, we ensure $\mean[X^4]\leq C<\infty$ and that $h_0(\theta)=\mean_\theta[X]$ is Lipschitz continuous which implies \ref{ass:finite:fourth:moment} and \ref{ass:minimizer:exists}.

    A classic and well-known example that fits into this framework is the Black-Scholes model with volatility ranging for example between $10\%-20\%$, see \cite{BlackScholes1973}. We select the volatility as $(1+\theta)\cdot 0.1$ for $\theta\in[0,1]$. In this case, the underlying asset price follows a geometric Brownian motion with volatility $\theta$, initial price $x>0$, and a risk-free interest rate $r>0$. The payoff of a European call option is then given by:
    \begin{align*}
        X
    &:= \Big(xe^{Z_T}-K\Big)_+:=\max\Big\{xe^{Z_T}-K,0\Big\},
    \end{align*}
    where $K>0$ is the strike price, $T>0$ is the maturity, and 
    \begin{align*}
        Z_T&\sim \mathcal{N}\left( -\frac{1}{200}(1 + \theta)^2 T, \frac{(1 + \theta)^2}{100} T \right) \text{ under } P_\theta.
    \end{align*} 
    Since the norm of $X$ is bounded by a geometric Brownian motion, it immediately follows $\mean[X^4]<\infty$. Here, the price formula of the option is known, \cite[Eq. (13)]{BlackScholes1973},
    \begin{align*}
        h_0(\Theta)&=\mean[X\cond{}\Theta] \\
        &=x  \Phi\left(d_1\left(\frac{1+\Theta}{10}\right)\right) - K  e^{-rT} \Phi\left(d_1\left(\frac{1+\Theta}{10}\right)-\frac{(1+\Theta)\sqrt T}{10}\right)
    \end{align*}
    with $\Phi$ the cumulative distribution function of $\mathcal N(0,1)$ and 
    \begin{align*}
			d_1(\sigma) ~&:=~ \frac{\ln \left(\frac{x}{K}\right) + \left(r+\frac1{2}\sigma^2\right) T}{\sigma \sqrt{T-t}}.
    \end{align*}
    The function $h_0'$ is continuous and its absolute value is bounded on $[0,1]$, hence $h_0$ is Lipschitz continuous. 

    Obviously, this setting works for any volatility interval $[a,b] \subset [0,\infty)$ by choosing the volatility $a+\theta(b-a)$ for $\theta\in[0,1]$. In addition, we can choose a more general European or even path dependent payoff.
\end{example}

\begin{example}
    We consider the Sobolov space 
    $$H^\alpha(\R^n):=\left\{f\in \mathcal L^2(\lambda^n)\cond{\big} \int_{\R^n} (1+|x|^\alpha) |\widehat{f}(x)|^2d \lambda^n(x)<\infty\right\}$$
    with scalar product 
    $$ \langle f,g\rangle_\alpha:=(2\pi)^{-\frac n2}\int_{\R^n} (1+|x|^\alpha) \widehat{f}(x)\overline{\widehat{ g}}(x) d \lambda^n(x)$$
    where $\widehat f$ represents the Fourier transform for integrable $f$, that is, 
    $$\widehat f(x) := (2\pi)^{-\frac n2}\int_{\R^n} f(y) e^{i\langle x,y\rangle} d\lambda^n(y)$$
    and $\lambda^n$ the Lebesgue measure on $\R^n$.

    The function $k_x:\R^n\to \R$ is given via 
    \begin{align*}
        \widehat k_x(u):=\frac{e^{i\langle u,x\rangle}}{1+|u|^\alpha}.
    \end{align*}
    For $\alpha>\frac n2$, $x\in\mathbb R^n$ it holds that $\widehat k_x\in \mathcal L^2(\lambda^n)$ and one finds $k_x\in H_\alpha$ and for $f\in H_\alpha$ that
     $$ \< f, k_x \>_{\alpha} = f(x). $$
    Due to Sobolev embedding theorem \cite[Ch. 6]{adams2003sobolev},  $H^\alpha$ contains $C^\alpha(\R^n,\R)$ with compact support 
    for $\alpha\geq\frac n2$. 

    If $n=1$ and $\alpha=2$, then we find
     $$ k_x(y) = c e^{-|x-y|}$$
    for some constant $c>0$ because $u\mapsto \frac 1{\pi(1+|u|^2)}$ is the density of the Cauchy distribution.

    The function $k$ is again translation-invariant, hence \ref{ass:k:finite:secondmoment} holds, see \Cref{lem:translationinvariant_kernel}. We then have to choose a measure $\mu$ with support $\R^n$, i.e., \ref{ass:support:mu} and a random variable $X$ such that \ref{ass:finite:fourth:moment} and \ref{ass:minimizer:exists} apply. 
\end{example}
\begin{example} The space presented in Filipovi{\'{c}} \cite{filipovic} is another convenient RKHS. For a continuously differentiable and non-decreasing function $w : [0,\infty)\to [1,\infty)$ with $w(0)=1$. Let now $H$ be the set of all absolutely continuous functions $f: [0,\infty) \to \mathbb{R}$ for which
\begin{equation}
	\label{filipoviccond}
	\int_{0}^\infty \big(f'(x)\big)^2w(x)dx < \infty .
\end{equation}
With the inner product defined by
\begin{equation*}
	\langle f,g\rangle := f(0)g(0)+\int_{0}^\infty f'(x)g'(x)w(x)dx, \qquad \mbox{for } f, g\in H_w,
\end{equation*}
and $\norm{f}^2:= \langle f,f\rangle$, this space is a separable Hilbert space with $H_w\subset C([0,\infty),\R)$, see \cite[Theorem 5.1.1]{filipovic}. Moreover, $ f(x)=\langle f,k_x \rangle$ and $k_x\in H_w$ with $k_x$ defined by $$k(x,y):=k_x(y):=1+\int_0^{y \wedge x} \big(w(z)\big)^{-1} dz.$$ 

In this setting, we have $\Gamma:=[0,\infty)$ with some probability measure $\mu$ with support $\Gamma$, i.e., \ref{ass:support:mu}. If $w^{-1}$ is chosen to be integrable, then $k_x(y)\leq C<\infty$ and \ref{ass:k:finite:secondmoment} is true. Finally, we choose a payoff $X$ such that \ref{ass:finite:fourth:moment} and \ref{ass:minimizer:exists} are valid.

Similar to Example \ref{exm:[0,1]} the Black-Scholes model with volatility $\theta\in[0,\infty)$ fits into this setting. Using the same notation, the payoff is
\begin{align*}
    X:=\Big(xe^{Z_T}-K\Big)_+
\end{align*}
with 
\begin{align*}
    Z_T\sim \mathcal{N}\left(-\frac{1}{2} \theta^2 T,\; \theta^2 T\right) \text{ under } P_\theta
\end{align*}
and corresponding price formula 
\begin{align*}
    h_0(\Theta)=\mean[X\cond{}\Theta] =x  \Phi\big(d_1(\Theta)\big) - K  e^{-rT} \Phi\left(d_1(\Theta)-\Theta\sqrt T\right).
\end{align*}
For $w(x):=e^x$ we have $h_0\in H_w$, and $\mean[X^4]<\infty$ as $X$ is bounded by a geometric Brownian motion. For example, choosing $\mu$ as the exponential distribution, this example fulfills all assumptions. 
\end{example}

\begin{example}
\label{exm:featuremap}
   A feature map is any function 
   $$ \phi:\Gamma \rightarrow H, $$
   where $\Gamma$ is some set and $H$ any Hilbert space. 
   
   In the following, we use the specific Hilbert space $H$ given by $H=\mathcal L^2(\mathbb R^d,R)$ for some chosen probability measure $R$. For $x\in\Gamma$ we write $\phi_x$ for the image of $x$ under $\phi$ and treat it as a random variable, that is, we consider $\E_R[XY]$ for the scalar product of $X,Y\in H$. Let $\sigma:\mathbb R\rightarrow \mathbb R$ be measurable and bounded. We define 
    $$ \phi_x:=\sigma(x+\cdot) \in H$$
    for any $x\in\mathbb R^d$ where $\sigma$ is applied component-wise. Then, $\phi$ is a feature map. Recall that $k(x,y):=\E_R[\phi_x\phi_y]$ for $x,y\in \Gamma$ defines a kernel on $\Gamma$, cf.\ \cite{aronszajn1950theory}. For $X\in H$ we define $f_X:\Gamma\rightarrow\mathbb R,x\mapsto \E_R[X\phi_x]$ and note that 
 $$ \Phi:H\rightarrow H_k, X\mapsto f_X $$
  is a surjective, continuous, linear map from $H$ to $H_k$.
  
    Observe that
   $$ \vert f\vert^2_{H_k} = \inf_{X\in H,f_X=f} \E_R[X^2].$$

   Further, we consider a generic quadratic function on $H_k$
    $$ q(f):=\<Qf,f\>_{H_k}-2\<a,f\>_{H_k}+c$$
    where $Q$ is positive-semidefinite on $H_k$, $a\in H_k$ and $c\in\mathbb R$, and assume that $f_0$ is a minimizer for $q$.

   We define the quadratic function
    $$ \bar q(X):=q(\Phi(X)), \quad X\in H.$$
   Note that $\bar q$ has a minimizer and any minimizer $X_0$ of $\bar q$ satisfies $\Phi(X_0)=f_0$. This allows to pull the minimization problem of $q$ in $H_k$ to $\bar q$ in $H$. The latter has a known norm which might be easier implemented in some cases. 

   Additionally, if the function $f\mapsto q(f)+\lambda \|f\|_{H_k}^\alpha$ is to be minimized over $f$ for some fixed $\alpha\geq 1$, then the unique minimizer $X_0$ of
    $$ X\mapsto \bar q(X)+\lambda \|X\|^\alpha_H $$
    satisfies that $\Phi(X_0)$ is a minimizer of $f\mapsto q(f)+\lambda \|f\|_{H_k}^\alpha$ due to the norm representation.
   
   Note though, kernel ridge regression can also be implemented without knowledge of the norm on the RKHS but relies on the (numerical) solution $f_0$ to $(Q+\lambda)f_0=a$ on the finite dimensinal subspace generated by $k_{x_1},\dots,k_{x_N}$ where $x_1,\dots,x_N\in \Gamma$ are the observation points, see \cite[Theorem 1]{Schoelkopf2001representer}  for details. 
   If the number of observations $N$ is high, this is a technical problem on its own.

   An approximation of kernel ridge regression via neural networks and gradient descent on the other hand requires the implementation of the norm of the given space for the penalty term. Here, it is useful to use a generic Hilbert space like a standard $\Lc^2$-space where many known techniques are available for implementation of the norm. We shall make use of this observation in our case study. 

In the above setting of a feature map with $H=\mathcal L^2(\mathbb R^d,R)$ we now consider a Heston model for the stock price:
\begin{eqnarray}
dS_t &=& \mu S_t \, dt + \sqrt{v_t} S_t \, dW_t^S, \nonumber\\
dv_t &=& \kappa(\theta - v_t) dt + \sigma \sqrt{v_t} \, dW_t^v,\nonumber \\
dW_t^S dW_t^v &=& \rho dt,\nonumber
\end{eqnarray}
where $S_t$ is the stock price, and $\sqrt{v_t}$ the instantaneous volatility. The parameter $\kappa$ determines the mean reversion speed to the long term variance $\theta$, $\sigma$ the volatility of the volatility (vol of vol), and $\rho$ the constant correlation between the two Brownian motions $ W^S $ and $W^v$. In our case study in Section~\ref{sec:case_study}, we fix a payoff function and then aim to learn the map $(\kappa, \theta, \sigma, \rho)\mapsto P_{\kappa, \theta, \sigma, \rho}$, where $P_{\kappa, \theta, \sigma, \rho}$ is the price of the payoff under the Heston model with the parameters $\kappa, \theta, \sigma, \rho$.
\end{example}

\section{Proofs of the main results}
\label{sec:proofs}
\normalfont
As a tool, we introduced a real reproducing kernel Hilbert space (RKHS). For a linear operator $T\in L(H)$ and $\lambda\in\mathbb R$ we understand by $T+\lambda$ the bounded operator with $(T+\lambda)h:=Th+\lambda h$. For $h\in H$ we define the operator $h\otimes h\in L(H)$ via
 $$ (h\otimes h)f := h\<h,f\>,\quad f\in H.$$ We rewrite the quadratic functions $V$ and $V_N$ in terms of operator theory and the kernel $k$. Hence, we make the specifications $Q_N:=\frac1N\sum_{i=1}^Nk_{\Theta_i}\otimes k_{\Theta_i}$, $Q:=\E[Q_1]$, $a_N:=\frac1N\sum_{i=1}^NX_ik_{\Theta_i}$ and $a:=\E[a_1]$. Thus, 
\begin{align*}
    V(h) &= \E[ (h(\Theta)-X)^2 ]\\
    &= \<Qh, h\> - 2\<a,h\> + \E[X^2]  \\
    V_N(h) &= \frac1N\sum_{i=1}^N (h(\Theta_i)-X_i)^2\\
    &= \<Q_Nh, h\> - 2\<a_N,h\> + \frac1N\sum_{i=1}^NX_i^2.
\end{align*}
Due to \Cref{l:Q properties} we find that $Q\in L(H)$ and $Q_N \in \Lc^2(\Pbb,L(H))$ are positive semidefinite operators on $H$ which are trace-class and Hilbert Schmidt. Observe that, $a\in H$ and $a_N\in \Lc^2(\Pbb,H)$. In the previous section, we introduced a regularized version by adding a small ridge term $\lambda \|h\|$. Using the operator representation of $V$, we now show that the minimum of the regularized version remains approximately the same for small $\lambda$.  
\begin{proposition}\label{p:0 approx lambda}
    The unique minimizer of $V$ is $h_0$. Moreover, the minimizer $h_\lambda$ for $V_\lambda$ exists for $\lambda >0$ and we have
     $$ \limitlambda h_\lambda = h_0.$$
\end{proposition}
\begin{proof}
    We observe
     $$ \nabla V(h) = 2 (Qh - a) = 2\E[k_\Theta (h(\Theta) - X)],\quad h\in H. $$
    Conditioning on $\Theta$ implies
      $$ \nabla V(h_0) = \bf0.$$
    Let $\tilde h_0\in H$ be a minimizer for $V$. Then, we have
     $$ {\bf0}= \nabla V(\tilde h_0) = 2\E[k_\Theta (\tilde h_0(\Theta) - X)] = 2\E[k_\Theta (\tilde h_0(\Theta) - h_0(\Theta))]$$
     where the last equality is obtained via conditioning on $\Theta$. Consequently, we derive
      $$ 0 = \< \tilde h_0-h_0, \nabla V(\tilde h_0)\> = 2 \E[(\tilde h_0(\Theta) - h_0(\Theta))^2] $$
      and thus $\tilde h_0=h_0$ $\Pbb^{\Theta}$-a.s. \Cref{l:as is equal} yields $\tilde h_0=h_0$. Due to \Cref{l:Q properties} $Q$ has a trivial kernel. \Cref{p:h to h0} implies $h_\lambda \rightarrow h_0$ as $\lambda \rightarrow 0$.
\end{proof}

To prove the main result, \Cref{t:approximation}, we first need a few auxiliary statements, such as the following Monte Carlo bounds:
\begin{lemma}\label{l:MC estimate}
    For every $N\in\mathbb N$ we have
    \begin{align*}
        \E[ \|a_N-a\|^2] &= \frac{\E[\|a_1-a\|^2]}{N}, \\
        \E[ \|Q_N-Q\|^2_{HS}] &= \frac{\E[\|Q_1-Q\|_{HS}^2]}{N}.        
    \end{align*}
\end{lemma}
\begin{proof}
    The statement follows from \Cref{l:MC in H} and \Cref{c:MC in HS}.
\end{proof}

Next, we show that the ridge regression $V_\lambda$ can be solved approximately by the random ridge regression $V_{N,\lambda}$ and quantify the $\Lc^1(\Pbb,H)$ error.
\begin{proposition}\label{p:random approximation}
    Let $\lambda \in(0,1]$. The minimizer $h_{N,\lambda}\in \Lc^2(\Pbb,H)$ of the random quadratic function $V_{N,\lambda}$ exists and there is a constant $C>0$ (not depending on $N$ or $\lambda$) such that
     $$ \E[\| h_{N,\lambda} - h_\lambda\|] \leq \frac{C}{\lambda^2 \sqrt{N}}. $$
\end{proposition}
\begin{proof}
    Lemma \ref{l:minimizer} yields existence of the minimizer and
    \begin{align*}
        h_\lambda &= (Q+\lambda)^{-1}a, \\
        h_{N,\lambda} &=(Q_N+\lambda)^{-1}a_N.
    \end{align*}
    Using \Cref{l:difference of inverse} we obtain
     $$ \|(Q_N+\lambda)^{-1}-(Q+\lambda)^{-1}\|_{\op} \leq \frac{\|Q_N-Q\|_{\op}}{\lambda^2}.$$
    Thus, we find
     \begin{align*}
         \|h_{N,\lambda} - h_\lambda\| &\leq \|((Q_N+\lambda)^{-1}-(Q+\lambda)^{-1})a_N\| + \|(Q+\lambda)^{-1}(a_N-a)\| \\
           &\leq \|((Q_N+\lambda)^{-1}-(Q+\lambda)^{-1})\|_{\op}\|a_N\| + \|(Q+\lambda)^{-1}\|_{\op}\|a_N-a\| \\
           &\leq \frac{\|Q_N-Q\|_{\op}}{\lambda^2}\|a_N\|+\frac{\|a_N-a\|}\lambda,
     \end{align*} 
     where we used \Cref{l:difference of inverse,l:inverse norm} for the last inequality. 
     Using expectation and the Cauchy-Schwarz inequality yields
     \begin{align*}
         \E[\|h_{N,\lambda} - h_\lambda\| ] &\leq \frac1{\lambda^2} \sqrt{\E[ \|Q_N-Q\|^2_{\op}]}\sqrt{\E[\|a_N\|^2]}+ \frac{\E[\|a_N-a\|]}{\lambda} .
     \end{align*}
     Exploiting the inequality $\|\T\|_{\op}\leq \|T\|_{HS}$ which holds for any symmetric operator $T\in L(H)$ and \Cref{l:MC estimate} we derive
     \begin{align*}
         \E[\|h_{N,\lambda} - h_\lambda\| ] &\leq \frac{\sqrt{\E[ \|Q_1-Q\|^2_{HS}]}}{\sqrt{N}\lambda^2} \sqrt{\E[\|a_N\|^2]} + \frac{\sqrt{\E[\|a_1-a\|^2]}}{\sqrt{N}\lambda}.
     \end{align*}
     Since $\sqrt{\E[\|\cdot\|^2]}$ is a norm and due to \Cref{l:MC estimate} we have
      $$ \sqrt{\E[\|a_N\|^2]} \leq \sqrt{\E[\|a_N-a\|^2]} + \|a\| = \frac{\sqrt{\E[\|a_1-a\|^2]}}{\sqrt{N}} + \|a\|.$$
     Choosing
      $$ C:=  \sqrt{\E[ \|Q_1-Q\|^2_{HS}]}\left(\sqrt{\E[\|a_1-a\|^2]} + \|a\|\right) + \sqrt{\E[\|a_1-a\|^2]}$$
      yields the claim.
\end{proof}
Combining the previous results, we now prove the main result \Cref{t:approximation}.
\begin{proof}[Proof of \Cref{t:approximation}]
    Let $N\in\mathbb N$ and $\widetilde C>0$ be the constant from Proposition \ref{p:random approximation}. Proposition \ref{p:h to h0} yields $\|h_{\lambda_N}-h_0\|_{\mathcal L^2(\mu )} \leq \sqrt{\lambda_N} \|h_0\|/2$. Additionally using
     $ \|h\|_{\mathcal L^2(\mu )} = (\<Qh,h\>)^{1/2} \leq \|Q\|_{\op}^{1/2} \|h\|$, we conclude
     \begin{align*}
         \E[ \|h_{N,\lambda_N}-h_0\|_{\mathcal L^2(\mu )}] &\leq \E[ \|h_{N,\lambda_N}-h_{\lambda_N}\|_{\mathcal L^2(\mu )}] + \|h_{\lambda_N}-h_0\|_{\mathcal L^2(\mu )} \\
             &\leq \frac{\widetilde C\|Q\|_{\op}^{1/2}}{\lambda_N^2\sqrt{N}} + \frac{\sqrt{\lambda_N}\|h_0\|}{2} \\
             &\leq \Big(\widetilde C\|Q\|_{\op}^{1/2} +  \frac{\|h_0\|}{2}\Big) \max\big\{\lambda_N^{-2}N^{-\frac12},\sqrt{\lambda_N}\big\}
     \end{align*}  
     as required.
\end{proof}

\begin{proof}[Proof of \Cref{p:loss}]
    Lemma \ref{l:minimizer} yields existence of the minimizer and
    \begin{align*}
        h_{N,\lambda} &=(Q_N+\lambda)^{-1}a_N, \\
        V_{N}(h_{N,\lambda}) &= \frac1N\sum_{i=1}^NX_i^2 - \<a_N,h_{N,\lambda}\>.
    \end{align*}
   Combining $a_N\rightarrow a$ in $\Lc^2(\Pbb,H)$ with \ref{ass:finite:fourth:moment} implies $\frac1N\sum_{i=1}^NX_i^2 \rightarrow \E[X^2]$ in $\Lc^2(\Pbb)$, and hence we find $$\lim\limits_{N\to\infty}V_{N}(h_{N,\lambda_N}) = V(h_0)$$ in $\Lc^1(\Pbb)$.
   We also have
    $$ V(h_0) = \E[(h_0(\Theta)-X)^2] = \E[ (X^2-\E[X\cond{}\Theta]^2)] = \E[\Var(X\cond{}\Theta)].$$
\end{proof}
\section{Case study}\label{sec:case_study}
In this section, we present a simulation study to verify the results for MinMC numerically. We focus here on a European call option with strike $K=100$ and the initial value of the asset $x=100$. To compute the argmin of $V_{N,\lambda}$, we use a neural network. For the training, we use the gradient-based optimization algorithm Adam with 50 epochs, a batch size of 32, and a learning rate of 0.0001. For the network structure, we used the hyperbolic tangent as the activation function, with two hidden layers, and 200 nodes in each layer.

First, we take a closer look at the Black-Scholes model with volatility $10 \% - 20 \%$. In Example \ref{exm:[0,1]} we introduce a suitable theoretical setting for the MinMC framework. Recall that the volatility is $(1+\theta)\cdot 0.1$ for $\theta\in [0,1]$. Since the price formula is known, an approximation is not necessary, but it allows us to verify the accuracy of our approximation results. 

\begin{figure}[h]
    \centering
    \input{plotBlackScholes}
    \caption{Model approximations for a European call option in the Black-Scholes model for different values of $\lambda$, but based on the same samples, compared to the true price (Black–Scholes formula).}
    \label{fig:BS:approximationonsamesample}
\end{figure}
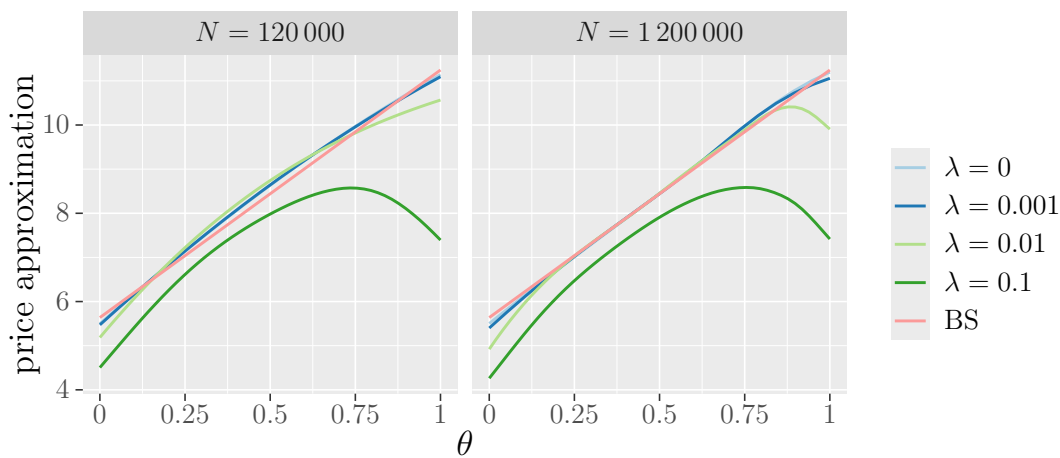
In \Cref{fig:BS:approximationonsamesample} we visualize the price approximation for different values of $\lambda$ based on the same samples $(X_1,\Theta_1),\dots,(X_N,\Theta_N)$ and the true price obtained using the Black-Scholes formula. A noticeable improvement is observed with decreasing $\lambda$. The results for $\lambda = 0$ and $\lambda = 0.001$ are nearly indistinguishable. Notably, $\lambda = 0$ produces excellent results, although we have not provided a formal proof in this case.
For both values of $N$, the simulation for $\lambda=0.1$ tends to significantly underestimate the price. Since the minimizer of $V_{N,\lambda}$ has a limit for every fixed $\lambda>0$, see \Cref{p:random approximation}, the results suggest that this limiting behavior is effectively reached at $N= 120\, 000$ for $\lambda=0.1$. As mentioned in \Cref{rem:convergencerate} we need $\lambda_N\to 0$ and $\lambda_NN^{\frac14}\to \infty$, hence choosing $\lambda < N^{-\frac14}$ seems reasonable. Here, $(120\,000)^{-\frac14}\approx 0.0542$ justifies this behavior. 
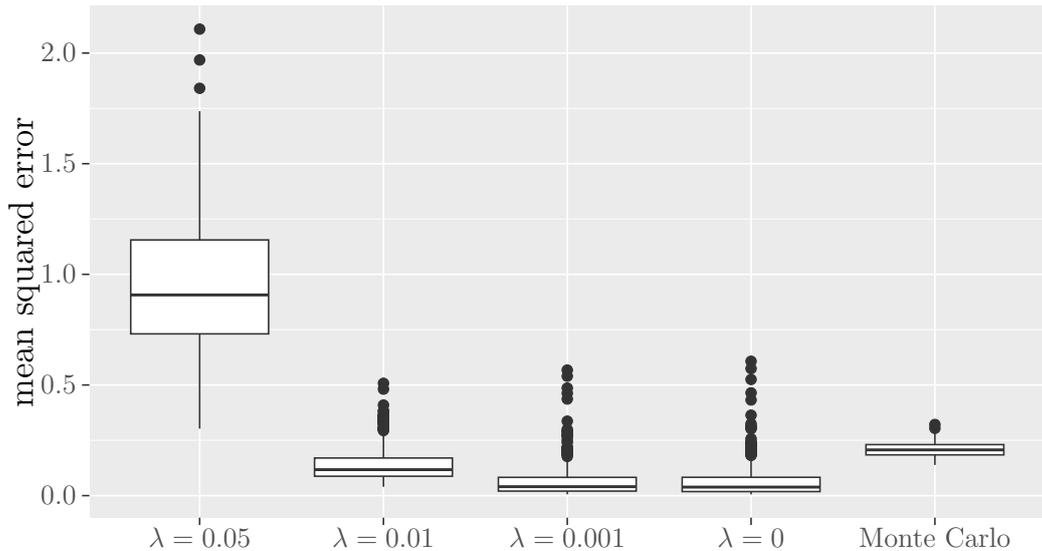
\begin{figure}[h]
	\centering
		\centering
		\input{boxplotdifferentlambda}
	\caption{Mean squared error in the Black-Scholes model based on $500$ simulation runs for $N=100\, 000$ compared to Monte Carlo simulations (each with $1\, 000$ samples) for $100$ $\theta$ values chosen equidistant in $[0,1]$.} 
    \label{fig:BS:mselambdacomparedtoMC}
\end{figure}
\begin{table}[h]
\centering
\begin{tabular}{|l|c|c|c|}
 \hline & median & 25\% quantile & 75\% quantile \\
\hline 
$\lambda=0.05$ & 0.9068& 0.7310&1.1554\\
$\lambda=0.01$&0.1173&0.0876&0.1700\\
$\lambda=0.001$&0.0404&0.0204&0.0826\\
$\lambda=0$& 0.0385&0.0183&0.0828\\
Monte Carlo &0.2066&0.1842 & 0.2305\\
\hline 
\end{tabular}
\caption{Median, 25\% and 75\% quantiles of the mean squared error obtained from 500 repeated price approximations for various $\lambda$ values and the Monte Carlo method. These values correspond to the boxplots shown in \Cref{fig:BS:mselambdacomparedtoMC}.}
\label{tab:mse_quantile_comparison}
\end{table}

Next, we fit the neural network $500$ times using different training sets each time, compute the corresponding price approximations, and calculate the mean squared error for various values of $\lambda$.
 In comparison, we perform Monte Carlo simulations with $1\,000$ samples for $100$ $\theta$ values, chosen equidistant in $[0,1]$. In all cases, the mean squared error is calculated with respect to the Black-Scholes model on these equidistant values. For the MinMC neural network, we choose samples $N=100\cdot 1\,000$ to ensure that the simulation results are based on the same total number of random variables. The resulting boxplots are presented in \Cref{fig:BS:mselambdacomparedtoMC}. The Monte Carlo simulation has the fewest outliers. This is not surprising since we calculate the mean squared errors on the 100 equidistant $\theta$ values and the Monte Carlo simulation is performed exactly for these values of $\theta$. The boxplots show again an improvement with decreasing $\lambda$. Although no visible improvement can be observed when comparing $\lambda = 0.001$ and $\lambda = 0$ in \Cref{fig:BS:mselambdacomparedtoMC}, the corresponding median and quantiles presented in \Cref{tab:mse_quantile_comparison} indicate a slight improvement. MinMC provides better results than the Monte Carlo method for $\lambda\in \{0.01,0.001,0\}$. 
\begin{figure}[h]
    \centering
    \centering
    \input{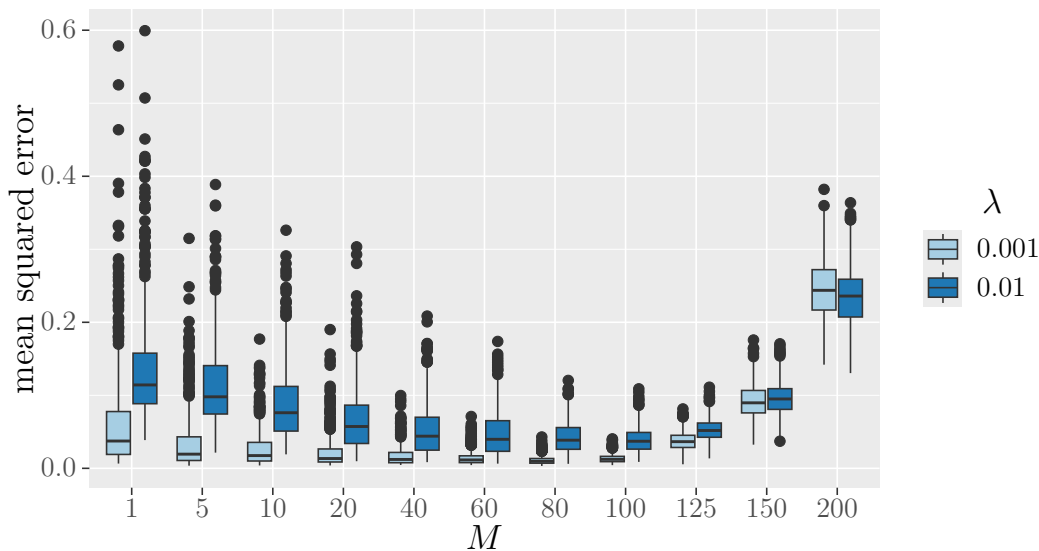}
\caption{Mean squared error in the Black-Scholes model based on $500$ simulation runs for $\lambda=0.01$ and $\lambda=0.001$ with fixed number of random variables $M\cdot N=120\, 000$. For fixed $M$ the simulations are based on the same samples.}
\label{fig:BS:msedifferentM}
\end{figure}

We now combine the MinMC setting with a Monte Carlo average to find the samples $(X_1,\Theta_1),\dots,(X_N, \Theta_N)$. For each model parameter $\Theta_i$ we sample $M$ i.i.d.\ random variables $Y_{i,1},\dots,Y_{i,M}$ from this model and define the average 
\begin{align}
\label{eq:preaverage}
    X_i:=\frac1M \sum\limits_{j=1}^MY_{i,j}.
\end{align} This results in an i.i.d.\ sequence $(X_1,\Theta_1),\dots,(X_N, \Theta_N)$ which satisfies the MinMC assumptions. The conditional variance of the $X_i$, which appears in Proposition \ref{p:loss}, is then lower as the conditional variance of $Y_{i,j}$. Increasing the number $M$ allows to reduce the observed mse-error with increasing $M$. We fix the number of random variables $M\cdot N$ and calculate the mean squared error for different values of $M$. The results of $500$ neural network fittings, each performed on a different training set, are visualized in \Cref{fig:BS:msedifferentM} for $\lambda=0.01$ and $\lambda=0.001$. Given that $\Theta_i$ is uniformly distributed over the interval $[0, 1]$, it can be assumed that some degree of averaging naturally occurs when considering a sufficiently large number of closely spaced values. In particular, MinMC performs well for $M = 1$. However, the boxplots clearly show that explicitly performing averaging yields improved results. The mean squared error initially decreases with increasing values of $M$, reaches an optimum in the range of $M = 60$ to $M = 125$, and subsequently increases as $M$ continues to increase. The smallest median in \Cref{fig:BS:msedifferentM} is $0.0371$ in the case $\lambda=0.01$ for $M=100$ and $0.0097$ in the case $\lambda=0.001$ for $M=80$.
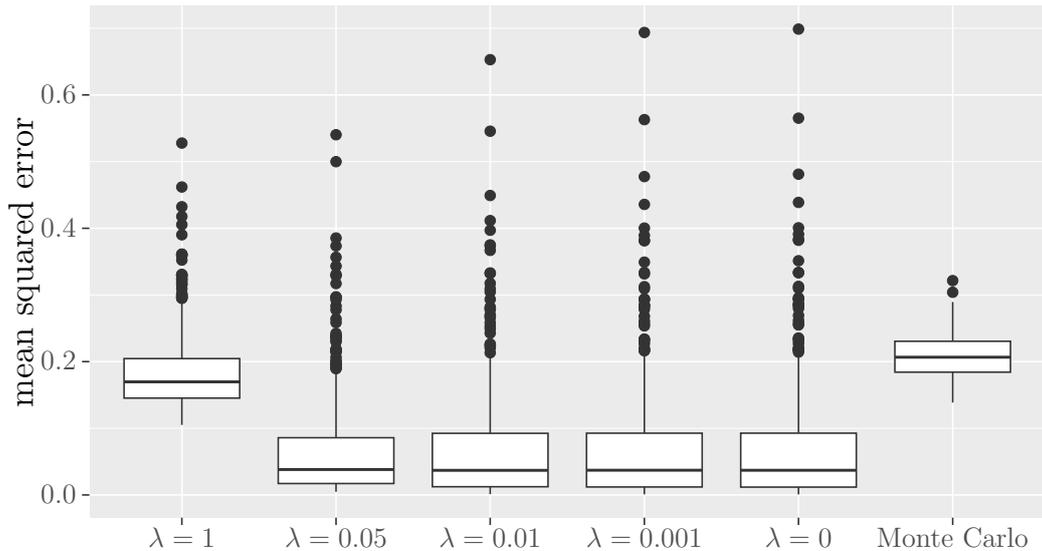
\begin{figure}[h]
		\centering		\input{boxplotdifferentlambdafeaturemaps}
	\caption{Mean squared error in the Black-Scholes model for the feature map setting based on $500$ simulation runs for $N=100\, 000$ compared to Monte Carlo simulations (each with $1\, 000$ samples) for $100$ $\theta$ values chosen equidistant in $[0.1,0.2]$.}
    \label{fig:BS:featuremap:mselambdacomparedtoMC}
\end{figure}
\begin{table}[h]
\centering
\begin{tabular}{|l|c|c|c|}
 \hline & median & 25\% quantile & 75\% quantile \\
\hline 
$\lambda=1$ & 0.1697&0.1453
&0.2045\\
$\lambda=0.05$ & 0.0381& 0.0171&0.0858\\
$\lambda=0.01$& 0.0369&0.0123&0.0924\\
$\lambda=0.001$&0.0371&0.0120&0.0927\\
$\lambda=0$& 0.0370&0.0118&0.0927\\
Monte Carlo &0.2066&0.1842 & 0.2305\\
\hline 
\end{tabular}
\caption{Median, 25\% and 75\% quantiles of the mean squared error obtained from 500 repeated price approximations for various $\lambda$ values and the Monte Carlo method. These values correspond to the boxplots shown in \Cref{fig:BS:featuremap:mselambdacomparedtoMC}.}
\label{tab:mse_quantile_comparison:featuremap}
\end{table}

In \Cref{exm:featuremap}, we introduced feature maps. By choosing $\sigma=\tanh$, the MinMC setting for the Black-Scholes model can be applied here as well. Since the setting is not restricted to the interval $[0,1]$ as in Example \ref{exm:[0,1]}, we consider the volatility $\theta\in [0.1,0.2]$ directly, without any rescaling. Analogous to \Cref{fig:BS:mselambdacomparedtoMC}, we compare the mean squared error for different $\lambda$ values based on 500 neural network fittings with Monte Carlo simulations in \Cref{fig:BS:featuremap:mselambdacomparedtoMC}. Here and in what follows, the mean squared error is computed on 100 equidistant points in the interval $[0.1,0.2]$. We observe that the feature maps yield strong performance even for larger values of $\lambda$, indicating that regularization does not need to be significantly reduced. In all presented cases, MinMC provides better results than Monte Carlo simulations based on the same total number of random variables. In particular, the results for $\lambda\in\{0.05,0.01,0.001,0\}$ are almost indistinguishable with the smallest median $0.0381$ achieved for $\lambda=0.05$. The medians, see \Cref{tab:mse_quantile_comparison:featuremap}, are slightly better than the optimal median of $0.0385$ for $\lambda=0$ in the previous simulations, see \Cref{tab:mse_quantile_comparison}. The outliers of MinMC are similarly scattered in both settings, if we disregard the case $\lambda=0.05$ in Figure~\ref{fig:BS:mselambdacomparedtoMC}.
\begin{figure}[h]
    \centering
\input{BlackScholesfeaturedifferentM}
\caption{Mean squared error in the Black-Scholes model for the feature map setting based on $500$ simulation runs for $\lambda=0.01$ and $\lambda=0.001$ with fixed number of random variables $M\cdot N=120\, 000$. For fixed $M$ the simulations are based on the same samples.}
\label{fig:BS:feature:msedifferentM}
\end{figure}

As before, we compute a Monte Carlo average over $M$ random variables for each model to obtain $N$ training samples with the total number of random variables fixed at $M\cdot N = 120\,000$, see \eqref{eq:preaverage}. The results for $\lambda=0.01$ and $\lambda=0.5$ are presented in \Cref{fig:BS:feature:msedifferentM}. In general, the mean squared error behaves similarly as in \Cref{fig:BS:msedifferentM}: it decreases with increasing $M$, reaches a minimum, and subsequently increases as $M$ becomes larger. However, the minimal medians in \Cref{fig:BS:feature:msedifferentM} are reached for significantly smaller values of $M$: 0.0535 at $M=10$ for $\lambda=0.5$ and 0.0095 at $M=30$ for $\lambda=0.01$. Although the optimal choice of $M$ differs depending on the specific setting, the behavior of the mean squared error at its optimum remains consistent. Overall, MinMC clearly shows superior performance compared to a classical Monte Carlo simulation. This is particularly surprising since the Monte Carlo simulations are performed exactly on the $100$ equidistant grid values in $[0,1]$ on which both MinMC and the classical Monte Carlo are benchmarked against the Black-Scholes formula.

Now, we consider the Heston model in this setting of feature maps, see Example \ref{exm:featuremap}. 
\begin{figure}[h]
    \centering
    \includegraphics[width=12cm]{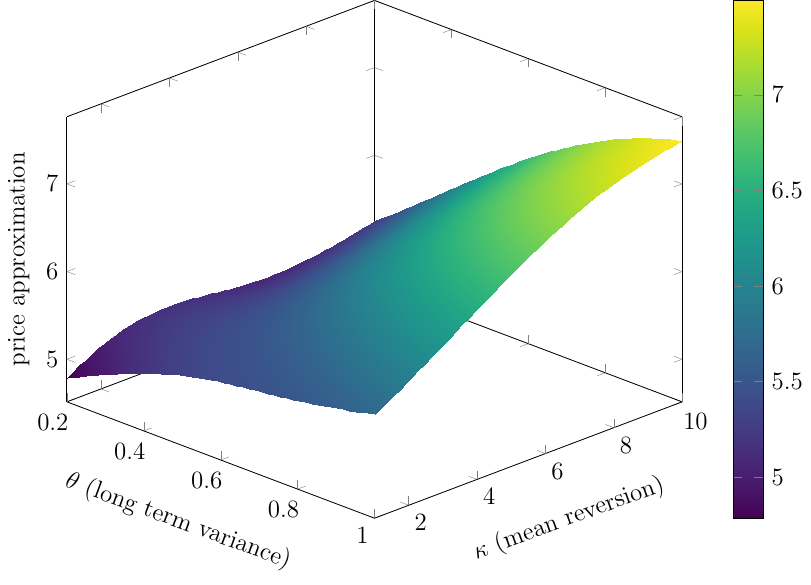}
    \caption{Heat map to visualize the MinMC approximation for a European call option price in the Heston model with $N=600\,000$ random variables, correlation $\rho=-0.5$ and volatility of volatility $\sigma=0.2$.}
    \label{fig:heatmap:Heston}
\end{figure}
We perform simulations under the Heston model in two different settings: in the first, we estimate the price function with respect to two parameters; in the second, with respect to three. In both cases, we fix the correlation at $\rho = -0.5$ and estimate the option price as a function of the mean reversion rate $\kappa \in [1, 10]$ and the long-term variance $\theta \in [0.2, 1]$. In the two-parameter case, we fix the volatility of volatility at $\sigma = 0.2$, whereas in the three-parameter case, $\sigma$ varies within the range $[0.2, 0.5]$. In Figure \ref{fig:heatmap:Heston} we visualize the price estimate by MinMC for $\lambda=0.001$ based on a sample with $N=600\,000$ random variables and maturity $T=1/12$ in the two-parameter case.
\begin{figure}[h]
    \centering
    \input{boxplot_Heston_Ktrain120000}
    \caption{Mean squared error in the Heston model based on $500$ simulation runs for various maturities with fixed number of random variables $M\cdot N=120\,000$. For fixed $M$ the simulations are based on the same samples.}
    \label{fig:boxplot:Heston:120k}
\end{figure}
\begin{figure}[h]
    \centering
    \input{boxplot_Heston_Ktrain600000}
    \caption{Mean squared error in the Heston model based on $500$ simulation runs for various maturities with fixed number of random variables $M\cdot N=600\,000$. For fixed $M$ the simulations are based on the same samples.}
    \label{fig:boxplot:Heston:600k}
\end{figure}

In contrast to the Black-Scholes model, the Heston model does not admit a closed-form solution for option prices in the general case. Consequently, to compute a benchmark for evaluating the mean squared error, we approximate the price using Monte Carlo simulation. In the two-parameter setting, we discretize each parameter interval into $100$ equidistant points, yielding a grid of $100^2$ parameter combinations. For each grid point, we estimate the benchmark price using $10^6$ Monte Carlo simulations. In the three-parameter case, we use $35$ equidistant points per dimension, resulting in $35^3$ total grid points, with the same number $10^6$ of Monte Carlo simulations per point. The mean-squared errors in the following are calculated based on these benchmarks and grids depending on the case (2 or 3 parameters). 

As in the Black-Scholes model, we perform pre-averaging using $M$ Monte Carlo simulations per parameter $\Theta_i$ to determine our $N$ random variables for the neural network fitting, see \eqref{eq:preaverage}. We fix the total number of random variables to $M\cdot N = 120\, 000$ or $M\cdot N = 600\,000$, respectively. The neural network is then trained $500$ times using different training sets in each run. For each trained model, we compute the corresponding price approximations and evaluate the mean squared error using the described benchmarks. The corresponding boxplots for different maturities are presented in Figures~\ref{fig:boxplot:Heston:120k} and~\ref{fig:boxplot:Heston:600k}. 

The first observation across all constellations is that the smaller the maturity, the smaller the mean squared error. Furthermore, for the same number of random variables, the case in which the function depends on two parameters is estimated more accurately than the case with three parameters. This is expected, as increasing the number of parameters complicates the approximation of the pricing function. This difference is particularly pronounced for $M\cdot N=120\,000$, see Figure \ref{fig:boxplot:Heston:120k}, but becomes significantly smaller for $M\cdot N=600\,000$ in Figure~\ref{fig:boxplot:Heston:600k}, where the overall error levels are notably lower. Perhaps in the first case, the total number of random variables is too small to estimate a function in three variables. In Figure~\ref{fig:boxplot:Heston:120k}, the smallest median values are attained at $M = 5$, although they are comparable to those observed for $M \in \{1, 10\}$, where the boxplots indicate lower variability. A similar observation can be made for Figure~\ref{fig:boxplot:Heston:600k}, where, however, the best median is achieved at $M=50$. Overall, the mean squared error in both figures behaves similar for the first three presented $M$ values and then increases monotonically.  Both figures suggest that having a greater number of grid points is more beneficial than using highly accurate grid points. A small value of $M$ may lead to some improvement. This is likely due to the fact that the function depends on multiple parameters. When $M \cdot N$ is too small, the function must first be explored across all relevant dimensions before more accurate evaluations contribute effectively to the estimation. Overall, also in this example, MinMC shows a strong performance gain compared to a classical Monte Carlo method.

\appendix
\section{Operator theoretic results}
In this section we provide all the results used from operator theory. We assume the same setting as in Section 2 and 3.

\begin{lemma}
\label{l:norm of h(theta)}
    Let $C:=\int_\Gamma \big(k(\theta,\theta)\big)^2\mu(d\theta)<\infty$ and $h\in H$. Then, we have
     $$ \E[h(\Theta)^4] \leq \|h\|^4 C <\infty.$$
\end{lemma}
\begin{proof}
    Using the Cauchy-Schwarz inequality and the definition of the kernel yield $|h(\Theta)|=|\< h,k_\Theta\>|\leq \|h\|\cdot \|k_\Theta\| = \|h\| \sqrt{k(\Theta,\Theta)}$. Thus, we get
    $$ \E[h(\Theta)^4]  \leq \|h\|^4 \E[(k(\Theta,\Theta))^2] = \|h\|^4 C $$
    where the last equality follows from $\Pbb^\Theta = \mu$.
\end{proof}

\begin{lemma}\label{l:as is equal}
 Let $h,g\in H$ with $h=g$ $\Pbb^\Theta$-a.s. Then, $h=g$.
\end{lemma}
\begin{proof}
    Define $A:=\{\gamma \in \Gamma: h(\gamma)=g(\gamma)\}$. By assumption, we find $\Pbb(\Theta\in A) = \Pbb^\Theta(A)=1$. Since $h,g$ are continuous, $A$ is closed. Assumption \ref{ass:support:mu} and $\Pbb^\Theta=\mu$ imply $A=\Gamma$, that is, $h=g$.
\end{proof}
\begin{lemma}\label{l:inverse norm}
    Let $\widetilde Q\in L(H)$ be positive semidefinite, trace-class. Then, $\widetilde Q+\lambda$ is invertible for any $\lambda >0$ and $(\widetilde Q+\lambda)^{-1}$ has operator norm at most $\lambda^{-1}$.
\end{lemma}
\begin{proof}
   Since $\widetilde Q$ is positive semidefinite with spectral values $\lambda_i\geq 0$, the operator $\widetilde Q+\lambda$ is positive definite with spectral values $\lambda_i+\lambda>0$. Hence, $\widetilde Q+\lambda$ is invertible and due to spectral mapping theorem \cite[2.1.10 Theorem]{palmer1994} its inverse has spectral values $(\lambda_i+\lambda)^{-1}$. Since the operator is self-adjoint and positive definite, the operator norm equals the supremum of the absolute values of the spectral values which is bounded by $\lambda^{-1}$, see 
   \cite[10.13 Theorem (b), 12.31 Theorem (a)]{rudin1991functional}.
\end{proof}
\begin{lemma}\label{l:strong_Q_convergence}
    Let $\widetilde Q\in L(H)$ be positive semidefinite, trace-class, and let $\Pi:H\rightarrow H$ denote the orthonormal projector on the kernel of $\widetilde Q$. Then, $\lambda (\widetilde Q+\lambda)^{-1}$ converges to $\Pi$ strongly as $\lambda\searrow 0$, i.e., for every $h\in H$ we have
     $$ \limitlambda\lambda (\widetilde Q+\lambda)^{-1}h = \Pi h.$$ 
    Moreover, if the kernel of $\widetilde Q$ is trivial, then
     $$ \limitlambda\lambda (\widetilde Q+\lambda)^{-1}h = \bf0.$$ 
\end{lemma}
\begin{proof}
      There exists an orthonormal basis $(e_i)_{i\in I}$ for some index set $I$ such that $e_i$ is an eigenvector of $\widetilde Q$ with some eigenvalue $\rho_i\geq 0$ for any $i\in I$ \cite[p. 905, 4 Theorem]{dunford1988linear_spectral_theory}. We find that
       \begin{align*}
           \lambda (\widetilde Q+\lambda)^{-1}h &= \sum_{i\in I} \<h,e_i\> \lambda (\widetilde Q+\lambda)^{-1}e_i \\
            &= \sum_{i\in I} \<h,e_i\> \lambda (\rho_i+\lambda)^{-1}e_i.
       \end{align*}
       Note that the following applies for each summand  
        $$ \limitlambda\lambda (\rho_i+\lambda)^{-1}e_i = \Pi e_i = \begin{cases} 0 & \rho_i>0, \\ e_i & \rho_i=0, \end{cases}$$ with $\|\lambda (\rho_i+\lambda)^{-1}e_i\|\leq 1$. Consequently, we get
         $$\limitlambda\sum_{i\in I} \<h,e_i\> \lambda (\rho_i+\lambda)^{-1}e_i = \sum_{i\in I} \<h,e_i\> \Pi e_i = \Pi h. $$
\end{proof}
\begin{lemma}\label{l:minimizer}
    Let $\widetilde Q \in L(H)$ be positive semidefinite, $\lambda>0$, $\widetilde a\in H$ and $c\in\mathbb R$. Then,
     $$ \widetilde V:H\rightarrow \mathbb R,\quad \widetilde V(h):= \<\widetilde Qh ,h\> - 2\<\widetilde a,h\> + c + \lambda \|h\|^2 $$
    has a unique minimizer given by
     $$ \widetilde h_0 := (\widetilde Q+\lambda)^{-1}\widetilde a $$
    with objective value
     $$ \widetilde V(\widetilde h_0) = c - \<\widetilde a,\widetilde h_0\>. $$
\end{lemma}
\begin{proof}
   $\widetilde V$ is convex and continuously differentiable. Consequently, $h\in H$ is a minimizer for $\widetilde V$ if and only if $\nabla \widetilde V(h)=\bf0$. Since $\nabla \widetilde V(h) = 2(\widetilde Q+\lambda)h-2\widetilde a$ we find that $h\in H$ is a minimizer if and only if $(\widetilde Q+\lambda)h=\widetilde a$. \Cref{l:inverse norm} yields invertibility of $\widetilde Q+\lambda$ and, thus, that the claimed element is the minimizer. Moreover, inserting the minimizer $\widetilde h_0$ into $\widetilde V$ implies the last equality.
\end{proof}
\begin{lemma}\label{l:difference of inverse}
    Let $\widetilde Q_1,\widetilde Q_2\in L(H)$ be positive semidefinite and $\lambda >0$. Then, we have
     $$ \|(\widetilde Q_1+\lambda)^{-1}-(\widetilde Q_2+\lambda)^{-1}\|_{\op} \leq \frac{\|\widetilde  Q_1-\widetilde Q_2\|_{\op}}{\lambda^2}.$$
\end{lemma}
\begin{proof}
    For $D:=\widetilde Q_2-\widetilde Q_1$ we derive
    \begin{align*}
        (\widetilde Q_1+\lambda)^{-1}-(\widetilde Q_2+\lambda)^{-1} &= (\widetilde Q_1+\lambda)^{-1}\left( 1 - (\widetilde Q_1+\lambda)(\widetilde Q_1+D+\lambda)^{-1}\right) \\
           &= (\widetilde Q_1+\lambda)^{-1}D(\widetilde Q_2+\lambda)^{-1}.
    \end{align*}
    Lemma \ref{l:inverse norm} yields the claim.
\end{proof}

\begin{lemma}
    Let $h\in H$. Then, $h\otimes h$ is positive semidefinite, trace-class and Hilbert Schmidt. We also have
     $$ \|h\otimes h\|_{\op} = \mathrm{Tr}(h\otimes h) = \|h\otimes h\|_{HS} = \|h\|^2. $$
\end{lemma}
\begin{proof} 
    If $h=0$, then the equality is trivial. Thus, we may assume that $h\neq 0$.

    There is an orthonormal basis $(b_n)_{n\in\mathbb N}$ such that $b_1:=\frac{h}{\|h\|}$. We find that
     $$ \mathrm{Tr}(h\otimes h) = \sum_{n\in\mathbb N}\<(h\otimes h)b_n,b_n\>= \sum_{n\in\mathbb N}\<h,b_n\>^2 = \<h,b_1\>^2=\|h\|^2. $$

    Since $(h\otimes h)^*=h\otimes h$ we get
     $$ \|h\otimes h\|_{HS}^2 = \mathrm{Tr}(\|h\|^2(h\otimes h))=\|h\|^4$$
     and, hence,
      $$ \|h\otimes h\|_{HS} = \|h\|^2. $$

   For the last equality, we have
    $$ \|h\otimes h\|_{\op} \leq \|h\|^2 $$
    and 
     $$ \|h\otimes h\|_{\op} \geq \|(h\otimes h)b_1\|=\|h\|^2$$
     as required.
\end{proof}

\begin{lemma}\label{l:Q properties}
    The random operators $Q_N:=\frac1N \sum_{k=1}^N k_{\Theta_k}\otimes k_{\Theta_k}$  for $N\in\mathbb N$ and the operator $Q:=\E[ k_\Theta\otimes k_\Theta] = \E[Q_1]$ are positive semidefinite, trace-class and Hilbert Schmidt. Moreover, $\mathrm{Tr}(Q),\|Q\|_{\op},\|Q\|_{HS}$ are all bounded by $\E[\|k_\Theta\|^2]$. If $f,g\in H$, then
     $$ \<Qf,g\> = \E[f(\Theta)g(\Theta)] = \<f,g\>_{\Lc^2(\mu)}. $$
    Furthermore, if the support of $\mu$ is $\Gamma$, then $Q$ is positive definite.
\end{lemma}
\begin{proof}
    $Q_N$ has obviously all the given properties. Note that
     $$ \|Q\|_{\op} \leq \E[ \|k_\Theta\|^2 ] < \infty $$
     which shows that $Q$ is well defined and positive semidefinite.
      We also know
     $$ \mathrm{Tr}(Q_1) = \| k_\Theta\|^2 $$
     and, consequently,  we find
      $$ \mathrm{Tr}(Q) = \E[\|k_\Theta\|^2 ] < \infty $$
      which yields that $Q$ is trace-class and, hence, Hilbert Schmidt. Recall that for a positive semidefinite operator $Q$ which is trace-class, we get the following inequality
 $$ \| Q\|_{HS}^2 \leq \|Q\|_{\op} \mathrm{Tr}(Q) $$
 and, thus, we obtained the stated upper bound for the trace-norm, the operator norm and the Hilbert Schmidt norm of $Q$.

    Now, assume that the support of $\mu$ is $\Gamma$ and let $h\in H$ with $Qh = \bf0$. We have
     \begin{align*}
         0 &= \< Qh, h\> = \E[ h(\Theta)^2 ].
     \end{align*}
     Consequently, $h(\Theta) =  \bf0$ $\Pbb^\Theta$-a.s. Since the support of $\mu=\Pbb^\Theta$ is $\Gamma$ we see that $h=\bf0$ on a dense set. \Cref{l:as is equal} implies $h=\bf0$ and hence the kernel of $Q$ is trivial.
\end{proof}
\begin{lemma}\label{l:f estimate}
    Let $\lambda \geq 0$ and define $f(x) := \frac{\lambda \sqrt{x}}{\lambda + x}$ for any $x\geq 0$. Then, $$f(x)\leq \sqrt{\lambda}/2$$ for any $x\geq 0$.
\end{lemma}
\begin{proof}
    Due to $f(0)=0\leq \sqrt{\lambda}/2$, we may assume $x>0$. Note that $0<\sqrt{x\lambda} \leq (x+\lambda)/2$ and, thus, we receive
     $$ f(x) \leq \sqrt{\lambda}/2$$
     as required.
\end{proof}
\begin{proposition}\label{p:no ridge estimate}
    Let $\widetilde Q\in L(H)$ be positive semidefinite. We define the semi-norm $\|f\|_{\widetilde Q} := \|\widetilde Q^{1/2}f\|$ for any $f\in H$. Then, for any $\lambda > 0$, $h\in H$ we get
     $$ \| \lambda (\widetilde Q+\lambda)^{-1}h\|_{\widetilde Q} \leq \sqrt{\lambda} \|h\|/2.$$
    Furthermore, if $\widetilde Q:=Q$ is the operator from Lemma \ref{l:Q properties}, then for any $\lambda > 0$, $h\in H$ we have
     $$ \| \lambda (Q+\lambda)^{-1}h\|_{\Lc^2(\mu)} \leq \sqrt{\lambda} \|h\|/2.$$
\end{proposition}
\begin{proof}
    We consider $f(x) := \frac{\lambda \sqrt{x}}{x+\lambda}$ and 
     $$f(\widetilde Q):= \lambda \widetilde Q^{1/2}(\widetilde Q+\lambda)^{-1}.$$
    Lemma \ref{l:f estimate} yields $f(x)\leq \sqrt{\lambda}/2$ for any $x\geq 0$. The spectral mapping theorem implies that $f(\widetilde Q)$ is bounded linearly with 
     $$ \|f(\widetilde Q)\|_{L(H)} \leq \sup\big\{ |f(x)| : 0\leq x \leq \|\widetilde Q\|_{L(H)}\big\} \leq \sqrt{\lambda}/2. $$
    We find
    \begin{align*}
        \| \lambda (\widetilde Q+\lambda)^{-1}h\|_{\widetilde Q} &= \| \lambda \widetilde Q^{1/2} (\widetilde Q+\lambda)^{-1}h\| \\
         &\leq \| f(\widetilde Q) \|_{L(H)} \|h\| \\
         &\leq \sqrt{\lambda} \|h\|/2
    \end{align*}
    which shows the first displayed inequality.
    
    Now, let $\widetilde Q$ be the operator from Lemma \ref{l:Q properties}, i.e., $\widetilde Q:=Q=\E[k_\Theta\otimes k_\Theta]$. Then, we have $\<Qf,g\>=\E[f(\Theta)g(\Theta)]$ for any $f,g\in H$. For $f\in H$ we compute
  \begin{align*}
      \| f \|^2_{\Lc^2(\mu)} &= \E[ f(\Theta)^2 ] \\
       &= \<Qf, f\> \\
       &= \<Q^{1/2}f,Q^{1/2}f\>  \\
       &= \|f\|_Q^2
  \end{align*}
  and the inequality results from the first part.  
\end{proof}
The following proposition shows that under some condition the optimizer for a non-degenerate regression problem, if exists, is obtained as the limit of the optimizers for the problem with an additional ridge term. 

\begin{proposition}\label{p:h to h0}
    Let $\widetilde Q\in L(H)$ be positive definite and trace-class, $\widetilde a\in H$ and $c\in\mathbb R$. Assume that $\widetilde h_0\in H$ is a minimizer of
     $$ \widetilde V:H\rightarrow \mathbb R,\quad \widetilde V(h):= \<\widetilde Qh ,h\> - 2\<\widetilde a,h\> + c. $$
    and $\widetilde h_\lambda$ is the minimizer of $\widetilde V_\lambda(h):=\widetilde V(h)+\lambda\|h\|^2$ for any $\lambda >0$ as given in Lemma \ref{l:minimizer}. Then, $$\limitlambda \widetilde h_\lambda = \widetilde h_0.$$
    Furthermore, if $\widetilde Q$ is the operator from Lemma \ref{l:Q properties}, we have
     $$ \|\widetilde h_\lambda - \widetilde h_0\|_{\Lc^2(\mu)} \leq \sqrt{\lambda} \|\widetilde h_0\|/2.$$
\end{proposition}
\begin{proof}
 Lemma \ref{l:minimizer} yields that $\widetilde h_\lambda = (\widetilde Q+\lambda)^{-1}a$ for any $\lambda >0$. Since $\widetilde h_0$ is a minimizer for $\widetilde V$ we find that
  $$ \nabla \widetilde V(\widetilde h_0) = \bf0.$$
  Due to $\nabla \widetilde V(h) = 2 (\widetilde Qh - \widetilde a)$ for any $h\in H$ we see that $\widetilde Qh_0 = \widetilde a$. Note that $(\widetilde Q+\lambda)^{-1}\widetilde Q = 1 - \lambda(\widetilde Q+\lambda)^{-1}$. For $\lambda>0$ we calculate
    \begin{align*}
        \widetilde h_\lambda &= (\widetilde Q+\lambda)^{-1}\widetilde Q\widetilde h_0, \\
                  &= \widetilde h_0 - \lambda(\widetilde Q+\lambda)^{-1}\widetilde h_0.
    \end{align*} \Cref{l:strong_Q_convergence} implies
     $$ \limitlambda \widetilde h_\lambda = \widetilde h_0.$$
    The bound on $\|\widetilde h_\lambda - \widetilde h_0\|_{\Lc^2(\mu)} = \|\lambda(\widetilde Q+\lambda)^{-1}\widetilde h_0\|_{\Lc^2(\mu)}$ is given in Proposition \ref{p:no ridge estimate}.
\end{proof}

The $\mathcal L^2(\mu)$-norm and the RKHS norm have a relation, namely $\widetilde Q^{1/2}$ is an isometrie:
\begin{lemma}\label{l:Q-root isometric}
    Let $\widetilde Q^{1/2}$ be the positive semidefinite square root of $\widetilde Q\in L(H)$. Then, we have for any $h\in H$ that $h\in \mathcal L^2(\mu)$ and
     $$ \|h\|_{\mathcal L^2(\mu)} = \|\widetilde Q^{1/2}h\|.$$
\end{lemma}
\begin{proof}
    Let $h\in H$. We find that 
     $$ \int_\Gamma h(\theta)^2 \mu(d\theta) = \E[h(\Theta)^2] = \<\widetilde Qh,h\> = \|\widetilde Q^{1/2}h\|^2 \leq \|\widetilde Q^{1/2}\|^2_{\mathrm{op}}\|h\|^2<\infty. $$
    Consequently, we find $h\in \mathcal L^2(\mu)$ and
     $$ \|h\|_{\mathcal L^2(\mu)} = \|\widetilde Q^{1/2}h\|_H.$$
\end{proof}

\section{Monte Carlo on a Hilbert space}

\begin{lemma}\label{l:MC in H}
    Let $(b_n)_{n\in\mathbb N}$ be identical distributed, square-integrable and uncorrelated 
    random variables on a Hilbert space. We have
     $$ \E\Big[ \Big\|\frac1N\sum_{n=1}^N b_n-\mu\Big\|^2 \Big] = \frac{\E[\|b_1-\mu\|^2]}{N}$$
     with $\mu:=\E[b_1]$.
\end{lemma}
\begin{proof}
    We derive
     $$ \Big\|\frac1N\sum_{n=1}^N b_n-\mu\Big\|^2 = \frac{1}{N^2}\sum_{n=1}^N\|b_n-\mu\|^2 + \frac1{N^2}\sum_{\underset{k\neq l}{k,l=1}}^N \<b_k-\mu,b_l-\mu\>. $$
    Taking expectation and using that the double-sum has zero expectation yields
     $$ \E\Big[ \Big\|\frac1N\sum_{n=1}^N b_n-\mu\Big\|^2 \Big] = \frac{\E[\|b_1-\mu\|^2]}{N}$$
     as required.
\end{proof}
In the following, we denote by $L_{HS}(H)$ the Hilbert Schmidt operators, see \cite[Section XI.6. Hilbert-Schmidt Operators]{dunford1988linear_spectral_theory}. 
\begin{corollary}\label{c:MC in HS}
    Let $(B_n)_{n\in\mathbb N}$ be identical distributed, square-integrable and uncorrelated $L_{HS}(H)$-valued random variables with $\E[\|B_1\|_{HS}^2]<\infty$. We have
     $$ \E\Big[ \Big\|\frac1N\sum_{n=1}^N B_n-\mu\Big\|_{HS}^2 \Big] = \frac{\E[\|B_1-\mu\|_{HS}^2]}{N}$$
     with $\mu:=\E[B_1]$.
\end{corollary}
\begin{proof}
    The set of Hilbert Schmidt operators $L_{HS}(H)$ with the Hilbert-Schmidt norm $\|\cdot\|_{HS}$ is a Hilbert space, see \cite[p. 1011, 4 Theorem]{dunford1988linear_spectral_theory}. The corresponding scalar product is $\<A,C\>_{HS} = \mathrm{Tr}(AC^*)$. Consequently, Lemma \ref{l:MC in H} can be applied to the sequence $(B_n)_{n\in\mathbb N}$ on this Hilbert space.
\end{proof}
\section*{Acknowledgment}
Computational support from the Zentrum für
Informations- und Medientechnologie (ZIM) at Heinrich Heine University is
gratefully acknowledged.
\bibliography{references}
\bibliographystyle{alpha}
\end{document}

%% file: plotBlackScholes.tex
\begin{tikzpicture}[x=1pt,y=1pt]
\definecolor{fillColor}{RGB}{255,255,255}
\path[use as bounding box,fill=fillColor,fill opacity=0.00] (0,0) rectangle (397.48,180.67);
\begin{scope}
\path[clip] (  0.00,  0.00) rectangle (397.48,180.67);
\definecolor{drawColor}{RGB}{255,255,255}
\definecolor{fillColor}{RGB}{255,255,255}

\path[draw=drawColor,line width= 0.6pt,line join=round,line cap=round,fill=fillColor] (  0.00,  0.00) rectangle (397.48,180.68);
\end{scope}
\begin{scope}
\path[clip] ( 31.71, 30.69) rectangle (171.98,158.60);
\definecolor{fillColor}{gray}{0.92}

\path[fill=fillColor] ( 31.71, 30.69) rectangle (171.98,158.60);
\definecolor{drawColor}{RGB}{255,255,255}

\path[draw=drawColor,line width= 0.3pt,line join=round] ( 31.71, 48.79) --
	(171.98, 48.79);

\path[draw=drawColor,line width= 0.3pt,line join=round] ( 31.71, 82.09) --
	(171.98, 82.09);

\path[draw=drawColor,line width= 0.3pt,line join=round] ( 31.71,115.39) --
	(171.98,115.39);

\path[draw=drawColor,line width= 0.3pt,line join=round] ( 31.71,148.69) --
	(171.98,148.69);

\path[draw=drawColor,line width= 0.3pt,line join=round] ( 54.03, 30.69) --
	( 54.03,158.60);

\path[draw=drawColor,line width= 0.3pt,line join=round] ( 85.91, 30.69) --
	( 85.91,158.60);

\path[draw=drawColor,line width= 0.3pt,line join=round] (117.79, 30.69) --
	(117.79,158.60);

\path[draw=drawColor,line width= 0.3pt,line join=round] (149.66, 30.69) --
	(149.66,158.60);

\path[draw=drawColor,line width= 0.6pt,line join=round] ( 31.71, 32.14) --
	(171.98, 32.14);

\path[draw=drawColor,line width= 0.6pt,line join=round] ( 31.71, 65.44) --
	(171.98, 65.44);

\path[draw=drawColor,line width= 0.6pt,line join=round] ( 31.71, 98.74) --
	(171.98, 98.74);

\path[draw=drawColor,line width= 0.6pt,line join=round] ( 31.71,132.04) --
	(171.98,132.04);

\path[draw=drawColor,line width= 0.6pt,line join=round] ( 38.09, 30.69) --
	( 38.09,158.60);

\path[draw=drawColor,line width= 0.6pt,line join=round] ( 69.97, 30.69) --
	( 69.97,158.60);

\path[draw=drawColor,line width= 0.6pt,line join=round] (101.85, 30.69) --
	(101.85,158.60);

\path[draw=drawColor,line width= 0.6pt,line join=round] (133.73, 30.69) --
	(133.73,158.60);

\path[draw=drawColor,line width= 0.6pt,line join=round] (165.60, 30.69) --
	(165.60,158.60);
\definecolor{drawColor}{RGB}{166,206,227}

\path[draw=drawColor,line width= 1.1pt,line join=round] ( 38.09, 57.06) --
	( 44.80, 63.01) --
	( 51.51, 68.86) --
	( 58.22, 74.61) --
	( 64.93, 80.25) --
	( 71.64, 85.79) --
	( 78.36, 91.22) --
	( 85.07, 96.53) --
	( 91.78,101.73) --
	( 98.49,106.81) --
	(105.20,111.77) --
	(111.91,116.61) --
	(118.62,121.34) --
	(125.34,125.94) --
	(132.05,130.42) --
	(138.76,134.78) --
	(145.47,139.03) --
	(152.18,143.15) --
	(158.89,147.15) --
	(165.60,151.04);
\definecolor{drawColor}{RGB}{31,120,180}

\path[draw=drawColor,line width= 1.1pt,line join=round] ( 38.09, 56.65) --
	( 44.80, 62.74) --
	( 51.51, 68.71) --
	( 58.22, 74.56) --
	( 64.93, 80.28) --
	( 71.64, 85.88) --
	( 78.36, 91.35) --
	( 85.07, 96.69) --
	( 91.78,101.90) --
	( 98.49,106.96) --
	(105.20,111.90) --
	(111.91,116.70) --
	(118.62,121.36) --
	(125.34,125.88) --
	(132.05,130.27) --
	(138.76,134.53) --
	(145.47,138.65) --
	(152.18,142.65) --
	(158.89,146.51) --
	(165.60,150.25);
\definecolor{drawColor}{RGB}{178,223,138}

\path[draw=drawColor,line width= 1.1pt,line join=round] ( 38.09, 51.88) --
	( 40.69, 54.98) --
	( 43.29, 58.02) --
	( 45.90, 61.00) --
	( 48.50, 63.92) --
	( 51.10, 66.79) --
	( 53.70, 69.59) --
	( 56.30, 72.34) --
	( 58.91, 75.03) --
	( 61.51, 77.66) --
	( 64.11, 80.23) --
	( 66.71, 82.74) --
	( 69.32, 85.19) --
	( 71.92, 87.58) --
	( 74.52, 89.92) --
	( 77.12, 92.20) --
	( 79.73, 94.42) --
	( 82.33, 96.58) --
	( 84.93, 98.69) --
	( 87.53,100.74) --
	( 90.14,102.74) --
	( 92.74,104.69) --
	( 95.34,106.58) --
	( 97.94,108.42) --
	(100.54,110.21) --
	(103.15,111.95) --
	(105.75,113.65) --
	(108.35,115.29) --
	(110.95,116.89) --
	(113.56,118.44) --
	(116.16,119.95) --
	(118.76,121.41) --
	(121.36,122.83) --
	(123.97,124.21) --
	(126.57,125.55) --
	(129.17,126.85) --
	(131.77,128.11) --
	(134.38,129.34) --
	(136.98,130.53) --
	(139.58,131.68) --
	(142.18,132.80) --
	(144.79,133.88) --
	(147.39,134.93) --
	(149.99,135.95) --
	(152.59,136.94) --
	(155.19,137.90) --
	(157.80,138.83) --
	(160.40,139.73) --
	(163.00,140.61) --
	(165.60,141.46);
\definecolor{drawColor}{RGB}{51,160,44}

\path[draw=drawColor,line width= 1.1pt,line join=round] ( 38.09, 40.54) --
	( 40.69, 43.62) --
	( 43.29, 46.70) --
	( 45.90, 49.77) --
	( 48.50, 52.82) --
	( 51.10, 55.83) --
	( 53.70, 58.79) --
	( 56.30, 61.70) --
	( 58.91, 64.53) --
	( 61.51, 67.30) --
	( 64.11, 69.97) --
	( 66.71, 72.56) --
	( 69.32, 75.05) --
	( 71.92, 77.45) --
	( 74.52, 79.74) --
	( 77.12, 81.94) --
	( 79.73, 84.03) --
	( 82.33, 86.03) --
	( 84.93, 87.94) --
	( 87.53, 89.76) --
	( 90.14, 91.50) --
	( 92.74, 93.16) --
	( 95.34, 94.75) --
	( 97.94, 96.27) --
	(100.54, 97.73) --
	(103.15, 99.13) --
	(105.75,100.46) --
	(108.35,101.73) --
	(110.95,102.91) --
	(113.56,104.01) --
	(116.16,105.01) --
	(118.76,105.90) --
	(121.36,106.68) --
	(123.97,107.32) --
	(126.57,107.82) --
	(129.17,108.15) --
	(131.77,108.29) --
	(134.38,108.21) --
	(136.98,107.90) --
	(139.58,107.34) --
	(142.18,106.50) --
	(144.79,105.40) --
	(147.39,104.03) --
	(149.99,102.40) --
	(152.59,100.55) --
	(155.19, 98.48) --
	(157.80, 96.24) --
	(160.40, 93.83) --
	(163.00, 91.29) --
	(165.60, 88.65);
\definecolor{drawColor}{RGB}{251,154,153}

\path[draw=drawColor,line width= 1.1pt,line join=round] ( 38.09, 59.40) --
	(165.60,152.79);
\end{scope}
\begin{scope}
\path[clip] (177.48, 30.69) rectangle (317.75,158.60);
\definecolor{fillColor}{gray}{0.92}

\path[fill=fillColor] (177.48, 30.69) rectangle (317.75,158.60);
\definecolor{drawColor}{RGB}{255,255,255}

\path[draw=drawColor,line width= 0.3pt,line join=round] (177.48, 48.79) --
	(317.75, 48.79);

\path[draw=drawColor,line width= 0.3pt,line join=round] (177.48, 82.09) --
	(317.75, 82.09);

\path[draw=drawColor,line width= 0.3pt,line join=round] (177.48,115.39) --
	(317.75,115.39);

\path[draw=drawColor,line width= 0.3pt,line join=round] (177.48,148.69) --
	(317.75,148.69);

\path[draw=drawColor,line width= 0.3pt,line join=round] (199.80, 30.69) --
	(199.80,158.60);

\path[draw=drawColor,line width= 0.3pt,line join=round] (231.67, 30.69) --
	(231.67,158.60);

\path[draw=drawColor,line width= 0.3pt,line join=round] (263.55, 30.69) --
	(263.55,158.60);

\path[draw=drawColor,line width= 0.3pt,line join=round] (295.43, 30.69) --
	(295.43,158.60);

\path[draw=drawColor,line width= 0.6pt,line join=round] (177.48, 32.14) --
	(317.75, 32.14);

\path[draw=drawColor,line width= 0.6pt,line join=round] (177.48, 65.44) --
	(317.75, 65.44);

\path[draw=drawColor,line width= 0.6pt,line join=round] (177.48, 98.74) --
	(317.75, 98.74);

\path[draw=drawColor,line width= 0.6pt,line join=round] (177.48,132.04) --
	(317.75,132.04);

\path[draw=drawColor,line width= 0.6pt,line join=round] (183.86, 30.69) --
	(183.86,158.60);

\path[draw=drawColor,line width= 0.6pt,line join=round] (215.73, 30.69) --
	(215.73,158.60);

\path[draw=drawColor,line width= 0.6pt,line join=round] (247.61, 30.69) --
	(247.61,158.60);

\path[draw=drawColor,line width= 0.6pt,line join=round] (279.49, 30.69) --
	(279.49,158.60);

\path[draw=drawColor,line width= 0.6pt,line join=round] (311.37, 30.69) --
	(311.37,158.60);
\definecolor{drawColor}{RGB}{166,206,227}

\path[draw=drawColor,line width= 1.1pt,line join=round] (183.86, 57.15) --
	(190.57, 62.55) --
	(197.28, 67.89) --
	(203.99, 73.17) --
	(210.70, 78.38) --
	(217.41, 83.51) --
	(224.12, 88.59) --
	(230.84, 93.61) --
	(237.55, 98.61) --
	(244.26,103.62) --
	(250.97,108.70) --
	(257.68,113.88) --
	(264.39,119.21) --
	(271.10,124.67) --
	(277.81,130.15) --
	(284.53,135.49) --
	(291.24,140.47) --
	(297.95,144.92) --
	(304.66,148.76) --
	(311.37,151.98);
\definecolor{drawColor}{RGB}{31,120,180}

\path[draw=drawColor,line width= 1.1pt,line join=round] (183.86, 55.45) --
	(190.57, 61.51) --
	(197.28, 67.38) --
	(203.99, 73.04) --
	(210.70, 78.50) --
	(217.41, 83.78) --
	(224.12, 88.88) --
	(230.84, 93.86) --
	(237.55, 98.76) --
	(244.26,103.65) --
	(250.97,108.64) --
	(257.68,113.81) --
	(264.39,119.21) --
	(271.10,124.80) --
	(277.81,130.36) --
	(284.53,135.60) --
	(291.24,140.22) --
	(297.95,144.08) --
	(304.66,147.19) --
	(311.37,149.65);
\definecolor{drawColor}{RGB}{178,223,138}

\path[draw=drawColor,line width= 1.1pt,line join=round] (183.86, 47.48) --
	(186.46, 51.26) --
	(189.06, 54.85) --
	(191.66, 58.27) --
	(194.27, 61.51) --
	(196.87, 64.59) --
	(199.47, 67.50) --
	(202.07, 70.27) --
	(204.67, 72.89) --
	(207.28, 75.38) --
	(209.88, 77.74) --
	(212.48, 79.99) --
	(215.08, 82.14) --
	(217.69, 84.20) --
	(220.29, 86.19) --
	(222.89, 88.12) --
	(225.49, 90.01) --
	(228.10, 91.86) --
	(230.70, 93.71) --
	(233.30, 95.56) --
	(235.90, 97.43) --
	(238.51, 99.33) --
	(241.11,101.28) --
	(243.71,103.26) --
	(246.31,105.27) --
	(248.91,107.31) --
	(251.52,109.36) --
	(254.12,111.40) --
	(256.72,113.42) --
	(259.32,115.41) --
	(261.93,117.37) --
	(264.53,119.30) --
	(267.13,121.19) --
	(269.73,123.06) --
	(272.34,124.92) --
	(274.94,126.79) --
	(277.54,128.69) --
	(280.14,130.64) --
	(282.75,132.58) --
	(285.35,134.44) --
	(287.95,136.11) --
	(290.55,137.45) --
	(293.16,138.39) --
	(295.76,138.87) --
	(298.36,138.81) --
	(300.96,138.17) --
	(303.56,136.91) --
	(306.17,135.06) --
	(308.77,132.81) --
	(311.37,130.49);
\definecolor{drawColor}{RGB}{51,160,44}

\path[draw=drawColor,line width= 1.1pt,line join=round] (183.86, 36.50) --
	(186.46, 39.82) --
	(189.06, 43.17) --
	(191.66, 46.52) --
	(194.27, 49.82) --
	(196.87, 53.07) --
	(199.47, 56.22) --
	(202.07, 59.27) --
	(204.67, 62.20) --
	(207.28, 65.01) --
	(209.88, 67.70) --
	(212.48, 70.26) --
	(215.08, 72.71) --
	(217.69, 75.06) --
	(220.29, 77.31) --
	(222.89, 79.48) --
	(225.49, 81.59) --
	(228.10, 83.63) --
	(230.70, 85.62) --
	(233.30, 87.55) --
	(235.90, 89.44) --
	(238.51, 91.28) --
	(241.11, 93.06) --
	(243.71, 94.78) --
	(246.31, 96.43) --
	(248.91, 98.00) --
	(251.52, 99.49) --
	(254.12,100.89) --
	(256.72,102.19) --
	(259.32,103.40) --
	(261.93,104.51) --
	(264.53,105.50) --
	(267.13,106.38) --
	(269.73,107.11) --
	(272.34,107.69) --
	(274.94,108.12) --
	(277.54,108.39) --
	(280.14,108.48) --
	(282.75,108.37) --
	(285.35,108.02) --
	(287.95,107.40) --
	(290.55,106.56) --
	(293.16,105.49) --
	(295.76,104.16) --
	(298.36,102.45) --
	(300.96,100.31) --
	(303.56, 97.77) --
	(306.17, 94.95) --
	(308.77, 92.00) --
	(311.37, 89.03);
\definecolor{drawColor}{RGB}{251,154,153}

\path[draw=drawColor,line width= 1.1pt,line join=round] (183.86, 59.40) --
	(311.37,152.79);
\end{scope}
\begin{scope}
\path[clip] ( 31.71,158.60) rectangle (171.98,175.17);
\definecolor{fillColor}{gray}{0.85}

\path[fill=fillColor] ( 31.71,158.60) rectangle (171.98,175.18);
\definecolor{drawColor}{gray}{0.10}

\node[text=drawColor,anchor=base,inner sep=0pt, outer sep=0pt, scale=  0.88] at (101.85,163.86) {$N = 120\, 000$};
\end{scope}
\begin{scope}
\path[clip] (177.48,158.60) rectangle (317.75,175.17);
\definecolor{fillColor}{gray}{0.85}

\path[fill=fillColor] (177.48,158.60) rectangle (317.75,175.18);
\definecolor{drawColor}{gray}{0.10}

\node[text=drawColor,anchor=base,inner sep=0pt, outer sep=0pt, scale=  0.88] at (247.61,163.86) {$N = 1\,200\,000$};
\end{scope}
\begin{scope}
\path[clip] (  0.00,  0.00) rectangle (397.48,180.67);
\definecolor{drawColor}{gray}{0.20}

\path[draw=drawColor,line width= 0.6pt,line join=round] ( 38.09, 27.94) --
	( 38.09, 30.69);

\path[draw=drawColor,line width= 0.6pt,line join=round] ( 69.97, 27.94) --
	( 69.97, 30.69);

\path[draw=drawColor,line width= 0.6pt,line join=round] (101.85, 27.94) --
	(101.85, 30.69);

\path[draw=drawColor,line width= 0.6pt,line join=round] (133.73, 27.94) --
	(133.73, 30.69);

\path[draw=drawColor,line width= 0.6pt,line join=round] (165.60, 27.94) --
	(165.60, 30.69);
\end{scope}
\begin{scope}
\path[clip] (  0.00,  0.00) rectangle (397.48,180.67);
\definecolor{drawColor}{gray}{0.30}

\node[text=drawColor,anchor=base,inner sep=0pt, outer sep=0pt, scale=  0.88] at ( 38.09, 19.68) {0};

\node[text=drawColor,anchor=base,inner sep=0pt, outer sep=0pt, scale=  0.88] at ( 69.97, 19.68) {0.25};

\node[text=drawColor,anchor=base,inner sep=0pt, outer sep=0pt, scale=  0.88] at (101.85, 19.68) {0.5};

\node[text=drawColor,anchor=base,inner sep=0pt, outer sep=0pt, scale=  0.88] at (133.73, 19.68) {0.75};

\node[text=drawColor,anchor=base,inner sep=0pt, outer sep=0pt, scale=  0.88] at (165.60, 19.68) {1};
\end{scope}
\begin{scope}
\path[clip] (  0.00,  0.00) rectangle (397.48,180.67);
\definecolor{drawColor}{gray}{0.20}

\path[draw=drawColor,line width= 0.6pt,line join=round] (183.86, 27.94) --
	(183.86, 30.69);

\path[draw=drawColor,line width= 0.6pt,line join=round] (215.73, 27.94) --
	(215.73, 30.69);

\path[draw=drawColor,line width= 0.6pt,line join=round] (247.61, 27.94) --
	(247.61, 30.69);

\path[draw=drawColor,line width= 0.6pt,line join=round] (279.49, 27.94) --
	(279.49, 30.69);

\path[draw=drawColor,line width= 0.6pt,line join=round] (311.37, 27.94) --
	(311.37, 30.69);
\end{scope}
\begin{scope}
\path[clip] (  0.00,  0.00) rectangle (397.48,180.67);
\definecolor{drawColor}{gray}{0.30}

\node[text=drawColor,anchor=base,inner sep=0pt, outer sep=0pt, scale=  0.88] at (183.86, 19.68) {0};

\node[text=drawColor,anchor=base,inner sep=0pt, outer sep=0pt, scale=  0.88] at (215.73, 19.68) {0.25};

\node[text=drawColor,anchor=base,inner sep=0pt, outer sep=0pt, scale=  0.88] at (247.61, 19.68) {0.5};

\node[text=drawColor,anchor=base,inner sep=0pt, outer sep=0pt, scale=  0.88] at (279.49, 19.68) {0.75};

\node[text=drawColor,anchor=base,inner sep=0pt, outer sep=0pt, scale=  0.88] at (311.37, 19.68) {1};
\end{scope}
\begin{scope}
\path[clip] (  0.00,  0.00) rectangle (397.48,180.67);
\definecolor{drawColor}{gray}{0.30}

\node[text=drawColor,anchor=base east,inner sep=0pt, outer sep=0pt, scale=  0.88] at ( 26.76, 29.11) {4};

\node[text=drawColor,anchor=base east,inner sep=0pt, outer sep=0pt, scale=  0.88] at ( 26.76, 62.41) {6};

\node[text=drawColor,anchor=base east,inner sep=0pt, outer sep=0pt, scale=  0.88] at ( 26.76, 95.71) {8};

\node[text=drawColor,anchor=base east,inner sep=0pt, outer sep=0pt, scale=  0.88] at ( 26.76,129.01) {10};
\end{scope}
\begin{scope}
\path[clip] (  0.00,  0.00) rectangle (397.48,180.67);
\definecolor{drawColor}{gray}{0.20}

\path[draw=drawColor,line width= 0.6pt,line join=round] ( 28.96, 32.14) --
	( 31.71, 32.14);

\path[draw=drawColor,line width= 0.6pt,line join=round] ( 28.96, 65.44) --
	( 31.71, 65.44);

\path[draw=drawColor,line width= 0.6pt,line join=round] ( 28.96, 98.74) --
	( 31.71, 98.74);

\path[draw=drawColor,line width= 0.6pt,line join=round] ( 28.96,132.04) --
	( 31.71,132.04);
\end{scope}
\begin{scope}
\path[clip] (  0.00,  0.00) rectangle (397.48,180.67);
\definecolor{drawColor}{RGB}{0,0,0}

\node[text=drawColor,anchor=base,inner sep=0pt, outer sep=0pt, scale=  1.10] at (174.73,  7.64) {$\theta$};
\end{scope}
\begin{scope}
\path[clip] (  0.00,  0.00) rectangle (397.48,180.67);
\definecolor{drawColor}{RGB}{0,0,0}

\node[text=drawColor,rotate= 90.00,anchor=base,inner sep=0pt, outer sep=0pt, scale=  1.10] at ( 13.08, 94.64) {price approximation};
\end{scope}
\begin{scope}
\path[clip] (  0.00,  0.00) rectangle (397.48,180.67);
\definecolor{fillColor}{RGB}{255,255,255}

\path[fill=fillColor] (328.75, 45.40) rectangle (391.98,143.89);
\end{scope}
\begin{scope}
\path[clip] (  0.00,  0.00) rectangle (397.48,180.67);
\definecolor{drawColor}{RGB}{0,0,0}

\end{scope}
\begin{scope}
\path[clip] (  0.00,  0.00) rectangle (397.48,180.67);
\definecolor{fillColor}{gray}{0.92}

\path[fill=fillColor] (334.25,108.72) rectangle (348.70,123.17);
\end{scope}
\begin{scope}
\path[clip] (  0.00,  0.00) rectangle (397.48,180.67);
\definecolor{drawColor}{RGB}{166,206,227}

\path[draw=drawColor,line width= 1.1pt,line join=round] (335.69,115.95) -- (347.26,115.95);
\end{scope}
\begin{scope}
\path[clip] (  0.00,  0.00) rectangle (397.48,180.67);
\definecolor{fillColor}{gray}{0.92}

\path[fill=fillColor] (334.25, 94.26) rectangle (348.70,108.72);
\end{scope}
\begin{scope}
\path[clip] (  0.00,  0.00) rectangle (397.48,180.67);
\definecolor{drawColor}{RGB}{31,120,180}

\path[draw=drawColor,line width= 1.1pt,line join=round] (335.69,101.49) -- (347.26,101.49);
\end{scope}
\begin{scope}
\path[clip] (  0.00,  0.00) rectangle (397.48,180.67);
\definecolor{fillColor}{gray}{0.92}

\path[fill=fillColor] (334.25, 79.81) rectangle (348.70, 94.26);
\end{scope}
\begin{scope}
\path[clip] (  0.00,  0.00) rectangle (397.48,180.67);
\definecolor{drawColor}{RGB}{178,223,138}

\path[draw=drawColor,line width= 1.1pt,line join=round] (335.69, 87.04) -- (347.26, 87.04);
\end{scope}
\begin{scope}
\path[clip] (  0.00,  0.00) rectangle (397.48,180.67);
\definecolor{fillColor}{gray}{0.92}

\path[fill=fillColor] (334.25, 65.36) rectangle (348.70, 79.81);
\end{scope}
\begin{scope}
\path[clip] (  0.00,  0.00) rectangle (397.48,180.67);
\definecolor{drawColor}{RGB}{51,160,44}

\path[draw=drawColor,line width= 1.1pt,line join=round] (335.69, 72.58) -- (347.26, 72.58);
\end{scope}
\begin{scope}
\path[clip] (  0.00,  0.00) rectangle (397.48,180.67);
\definecolor{fillColor}{gray}{0.92}

\path[fill=fillColor] (334.25, 50.90) rectangle (348.70, 65.36);
\end{scope}
\begin{scope}
\path[clip] (  0.00,  0.00) rectangle (397.48,180.67);
\definecolor{drawColor}{RGB}{251,154,153}

\path[draw=drawColor,line width= 1.1pt,line join=round] (335.69, 58.13) -- (347.26, 58.13);
\end{scope}
\begin{scope}
\path[clip] (  0.00,  0.00) rectangle (397.48,180.67);
\definecolor{drawColor}{RGB}{0,0,0}

\node[text=drawColor,anchor=base west,inner sep=0pt, outer sep=0pt, scale=  0.88] at (354.20,112.92) {$\lambda=0$};
\end{scope}
\begin{scope}
\path[clip] (  0.00,  0.00) rectangle (397.48,180.67);
\definecolor{drawColor}{RGB}{0,0,0}

\node[text=drawColor,anchor=base west,inner sep=0pt, outer sep=0pt, scale=  0.88] at (354.20, 98.46) {$\lambda=0.001$};
\end{scope}
\begin{scope}
\path[clip] (  0.00,  0.00) rectangle (397.48,180.67);
\definecolor{drawColor}{RGB}{0,0,0}

\node[text=drawColor,anchor=base west,inner sep=0pt, outer sep=0pt, scale=  0.88] at (354.20, 84.01) {$\lambda=0.01$};
\end{scope}
\begin{scope}
\path[clip] (  0.00,  0.00) rectangle (397.48,180.67);
\definecolor{drawColor}{RGB}{0,0,0}

\node[text=drawColor,anchor=base west,inner sep=0pt, outer sep=0pt, scale=  0.88] at (354.20, 69.55) {$\lambda=0.1$};
\end{scope}
\begin{scope}
\path[clip] (  0.00,  0.00) rectangle (397.48,180.67);
\definecolor{drawColor}{RGB}{0,0,0}

\node[text=drawColor,anchor=base west,inner sep=0pt, outer sep=0pt, scale=  0.88] at (354.20, 55.10) {BS};
\end{scope}
\end{tikzpicture}

%% file: boxplotdifferentlambda.tex
\begin{tikzpicture}[x=1pt,y=1pt]
\definecolor{fillColor}{RGB}{255,255,255}
\path[use as bounding box,fill=fillColor,fill opacity=0.00] (0,0) rectangle (397.48,216.81);
\begin{scope}
\path[clip] (  0.00,  0.00) rectangle (397.48,216.81);
\definecolor{drawColor}{RGB}{255,255,255}
\definecolor{fillColor}{RGB}{255,255,255}

\path[draw=drawColor,line width= 0.6pt,line join=round,line cap=round,fill=fillColor] (  0.00,  0.00) rectangle (397.48,216.81);
\end{scope}
\begin{scope}
\path[clip] ( 34.16, 18.22) rectangle (391.98,211.31);
\definecolor{fillColor}{gray}{0.92}

\path[fill=fillColor] ( 34.16, 18.22) rectangle (391.98,211.31);
\definecolor{drawColor}{RGB}{255,255,255}

\path[draw=drawColor,line width= 0.3pt,line join=round] ( 34.16, 47.41) --
	(391.98, 47.41);

\path[draw=drawColor,line width= 0.3pt,line join=round] ( 34.16, 89.15) --
	(391.98, 89.15);

\path[draw=drawColor,line width= 0.3pt,line join=round] ( 34.16,130.89) --
	(391.98,130.89);

\path[draw=drawColor,line width= 0.3pt,line join=round] ( 34.16,172.63) --
	(391.98,172.63);

\path[draw=drawColor,line width= 0.6pt,line join=round] ( 34.16, 26.54) --
	(391.98, 26.54);

\path[draw=drawColor,line width= 0.6pt,line join=round] ( 34.16, 68.28) --
	(391.98, 68.28);

\path[draw=drawColor,line width= 0.6pt,line join=round] ( 34.16,110.02) --
	(391.98,110.02);

\path[draw=drawColor,line width= 0.6pt,line join=round] ( 34.16,151.76) --
	(391.98,151.76);

\path[draw=drawColor,line width= 0.6pt,line join=round] ( 34.16,193.50) --
	(391.98,193.50);

\path[draw=drawColor,line width= 0.6pt,line join=round] ( 75.44, 18.22) --
	( 75.44,211.31);

\path[draw=drawColor,line width= 0.6pt,line join=round] (144.26, 18.22) --
	(144.26,211.31);

\path[draw=drawColor,line width= 0.6pt,line join=round] (213.07, 18.22) --
	(213.07,211.31);

\path[draw=drawColor,line width= 0.6pt,line join=round] (281.88, 18.22) --
	(281.88,211.31);

\path[draw=drawColor,line width= 0.6pt,line join=round] (350.70, 18.22) --
	(350.70,211.31);
\definecolor{drawColor}{gray}{0.20}
\definecolor{fillColor}{gray}{0.20}

\path[draw=drawColor,line width= 0.4pt,line join=round,line cap=round,fill=fillColor] ( 75.44,180.20) circle (  1.96);

\path[draw=drawColor,line width= 0.4pt,line join=round,line cap=round,fill=fillColor] ( 75.44,190.89) circle (  1.96);

\path[draw=drawColor,line width= 0.4pt,line join=round,line cap=round,fill=fillColor] ( 75.44,202.53) circle (  1.96);

\path[draw=drawColor,line width= 0.6pt,line join=round] ( 75.44,122.99) -- ( 75.44,171.56);

\path[draw=drawColor,line width= 0.6pt,line join=round] ( 75.44, 87.56) -- ( 75.44, 51.81);
\definecolor{fillColor}{RGB}{255,255,255}

\path[draw=drawColor,line width= 0.6pt,fill=fillColor] ( 49.64,122.99) --
	( 49.64, 87.56) --
	(101.25, 87.56) --
	(101.25,122.99) --
	( 49.64,122.99) --
	cycle;

\path[draw=drawColor,line width= 1.1pt] ( 49.64,102.24) -- (101.25,102.24);
\definecolor{fillColor}{gray}{0.20}

\path[draw=drawColor,line width= 0.4pt,line join=round,line cap=round,fill=fillColor] (144.26, 56.64) circle (  1.96);

\path[draw=drawColor,line width= 0.4pt,line join=round,line cap=round,fill=fillColor] (144.26, 53.80) circle (  1.96);

\path[draw=drawColor,line width= 0.4pt,line join=round,line cap=round,fill=fillColor] (144.26, 54.56) circle (  1.96);

\path[draw=drawColor,line width= 0.4pt,line join=round,line cap=round,fill=fillColor] (144.26, 51.64) circle (  1.96);

\path[draw=drawColor,line width= 0.4pt,line join=round,line cap=round,fill=fillColor] (144.26, 55.81) circle (  1.96);

\path[draw=drawColor,line width= 0.4pt,line join=round,line cap=round,fill=fillColor] (144.26, 53.60) circle (  1.96);

\path[draw=drawColor,line width= 0.4pt,line join=round,line cap=round,fill=fillColor] (144.26, 56.00) circle (  1.96);

\path[draw=drawColor,line width= 0.4pt,line join=round,line cap=round,fill=fillColor] (144.26, 57.07) circle (  1.96);

\path[draw=drawColor,line width= 0.4pt,line join=round,line cap=round,fill=fillColor] (144.26, 54.92) circle (  1.96);

\path[draw=drawColor,line width= 0.4pt,line join=round,line cap=round,fill=fillColor] (144.26, 51.85) circle (  1.96);

\path[draw=drawColor,line width= 0.4pt,line join=round,line cap=round,fill=fillColor] (144.26, 66.71) circle (  1.96);

\path[draw=drawColor,line width= 0.4pt,line join=round,line cap=round,fill=fillColor] (144.26, 54.56) circle (  1.96);

\path[draw=drawColor,line width= 0.4pt,line join=round,line cap=round,fill=fillColor] (144.26, 55.26) circle (  1.96);

\path[draw=drawColor,line width= 0.4pt,line join=round,line cap=round,fill=fillColor] (144.26, 52.12) circle (  1.96);

\path[draw=drawColor,line width= 0.4pt,line join=round,line cap=round,fill=fillColor] (144.26, 52.49) circle (  1.96);

\path[draw=drawColor,line width= 0.4pt,line join=round,line cap=round,fill=fillColor] (144.26, 51.71) circle (  1.96);

\path[draw=drawColor,line width= 0.4pt,line join=round,line cap=round,fill=fillColor] (144.26, 56.38) circle (  1.96);

\path[draw=drawColor,line width= 0.4pt,line join=round,line cap=round,fill=fillColor] (144.26, 54.52) circle (  1.96);

\path[draw=drawColor,line width= 0.4pt,line join=round,line cap=round,fill=fillColor] (144.26, 60.70) circle (  1.96);

\path[draw=drawColor,line width= 0.4pt,line join=round,line cap=round,fill=fillColor] (144.26, 51.08) circle (  1.96);

\path[draw=drawColor,line width= 0.4pt,line join=round,line cap=round,fill=fillColor] (144.26, 56.81) circle (  1.96);

\path[draw=drawColor,line width= 0.4pt,line join=round,line cap=round,fill=fillColor] (144.26, 53.64) circle (  1.96);

\path[draw=drawColor,line width= 0.4pt,line join=round,line cap=round,fill=fillColor] (144.26, 54.83) circle (  1.96);

\path[draw=drawColor,line width= 0.4pt,line join=round,line cap=round,fill=fillColor] (144.26, 68.91) circle (  1.96);

\path[draw=drawColor,line width= 0.4pt,line join=round,line cap=round,fill=fillColor] (144.26, 51.73) circle (  1.96);

\path[draw=drawColor,line width= 0.4pt,line join=round,line cap=round,fill=fillColor] (144.26, 56.27) circle (  1.96);

\path[draw=drawColor,line width= 0.4pt,line join=round,line cap=round,fill=fillColor] (144.26, 56.37) circle (  1.96);

\path[draw=drawColor,line width= 0.4pt,line join=round,line cap=round,fill=fillColor] (144.26, 58.46) circle (  1.96);

\path[draw=drawColor,line width= 0.4pt,line join=round,line cap=round,fill=fillColor] (144.26, 51.52) circle (  1.96);

\path[draw=drawColor,line width= 0.4pt,line join=round,line cap=round,fill=fillColor] (144.26, 58.31) circle (  1.96);

\path[draw=drawColor,line width= 0.4pt,line join=round,line cap=round,fill=fillColor] (144.26, 54.79) circle (  1.96);

\path[draw=drawColor,line width= 0.6pt,line join=round] (144.26, 40.73) -- (144.26, 50.70);

\path[draw=drawColor,line width= 0.6pt,line join=round] (144.26, 33.85) -- (144.26, 29.87);
\definecolor{fillColor}{RGB}{255,255,255}

\path[draw=drawColor,line width= 0.6pt,fill=fillColor] (118.45, 40.73) --
	(118.45, 33.85) --
	(170.06, 33.85) --
	(170.06, 40.73) --
	(118.45, 40.73) --
	cycle;

\path[draw=drawColor,line width= 1.1pt] (118.45, 36.33) -- (170.06, 36.33);
\definecolor{fillColor}{gray}{0.20}

\path[draw=drawColor,line width= 0.4pt,line join=round,line cap=round,fill=fillColor] (213.07, 42.90) circle (  1.96);

\path[draw=drawColor,line width= 0.4pt,line join=round,line cap=round,fill=fillColor] (213.07, 42.45) circle (  1.96);

\path[draw=drawColor,line width= 0.4pt,line join=round,line cap=round,fill=fillColor] (213.07, 46.64) circle (  1.96);

\path[draw=drawColor,line width= 0.4pt,line join=round,line cap=round,fill=fillColor] (213.07, 44.08) circle (  1.96);

\path[draw=drawColor,line width= 0.4pt,line join=round,line cap=round,fill=fillColor] (213.07, 43.24) circle (  1.96);

\path[draw=drawColor,line width= 0.4pt,line join=round,line cap=round,fill=fillColor] (213.07, 54.63) circle (  1.96);

\path[draw=drawColor,line width= 0.4pt,line join=round,line cap=round,fill=fillColor] (213.07, 47.61) circle (  1.96);

\path[draw=drawColor,line width= 0.4pt,line join=round,line cap=round,fill=fillColor] (213.07, 49.18) circle (  1.96);

\path[draw=drawColor,line width= 0.4pt,line join=round,line cap=round,fill=fillColor] (213.07, 42.28) circle (  1.96);

\path[draw=drawColor,line width= 0.4pt,line join=round,line cap=round,fill=fillColor] (213.07, 73.91) circle (  1.96);

\path[draw=drawColor,line width= 0.4pt,line join=round,line cap=round,fill=fillColor] (213.07, 62.98) circle (  1.96);

\path[draw=drawColor,line width= 0.4pt,line join=round,line cap=round,fill=fillColor] (213.07, 43.62) circle (  1.96);

\path[draw=drawColor,line width= 0.4pt,line join=round,line cap=round,fill=fillColor] (213.07, 44.46) circle (  1.96);

\path[draw=drawColor,line width= 0.4pt,line join=round,line cap=round,fill=fillColor] (213.07, 50.87) circle (  1.96);

\path[draw=drawColor,line width= 0.4pt,line join=round,line cap=round,fill=fillColor] (213.07, 44.23) circle (  1.96);

\path[draw=drawColor,line width= 0.4pt,line join=round,line cap=round,fill=fillColor] (213.07, 42.75) circle (  1.96);

\path[draw=drawColor,line width= 0.4pt,line join=round,line cap=round,fill=fillColor] (213.07, 49.13) circle (  1.96);

\path[draw=drawColor,line width= 0.4pt,line join=round,line cap=round,fill=fillColor] (213.07, 43.49) circle (  1.96);

\path[draw=drawColor,line width= 0.4pt,line join=round,line cap=round,fill=fillColor] (213.07, 51.53) circle (  1.96);

\path[draw=drawColor,line width= 0.4pt,line join=round,line cap=round,fill=fillColor] (213.07, 50.40) circle (  1.96);

\path[draw=drawColor,line width= 0.4pt,line join=round,line cap=round,fill=fillColor] (213.07, 41.31) circle (  1.96);

\path[draw=drawColor,line width= 0.4pt,line join=round,line cap=round,fill=fillColor] (213.07, 65.17) circle (  1.96);

\path[draw=drawColor,line width= 0.4pt,line join=round,line cap=round,fill=fillColor] (213.07, 43.40) circle (  1.96);

\path[draw=drawColor,line width= 0.4pt,line join=round,line cap=round,fill=fillColor] (213.07, 71.59) circle (  1.96);

\path[draw=drawColor,line width= 0.4pt,line join=round,line cap=round,fill=fillColor] (213.07, 44.23) circle (  1.96);

\path[draw=drawColor,line width= 0.4pt,line join=round,line cap=round,fill=fillColor] (213.07, 42.81) circle (  1.96);

\path[draw=drawColor,line width= 0.4pt,line join=round,line cap=round,fill=fillColor] (213.07, 49.37) circle (  1.96);

\path[draw=drawColor,line width= 0.4pt,line join=round,line cap=round,fill=fillColor] (213.07, 44.88) circle (  1.96);

\path[draw=drawColor,line width= 0.4pt,line join=round,line cap=round,fill=fillColor] (213.07, 42.36) circle (  1.96);

\path[draw=drawColor,line width= 0.4pt,line join=round,line cap=round,fill=fillColor] (213.07, 43.69) circle (  1.96);

\path[draw=drawColor,line width= 0.4pt,line join=round,line cap=round,fill=fillColor] (213.07, 44.40) circle (  1.96);

\path[draw=drawColor,line width= 0.4pt,line join=round,line cap=round,fill=fillColor] (213.07, 67.12) circle (  1.96);

\path[draw=drawColor,line width= 0.4pt,line join=round,line cap=round,fill=fillColor] (213.07, 42.32) circle (  1.96);

\path[draw=drawColor,line width= 0.4pt,line join=round,line cap=round,fill=fillColor] (213.07, 43.28) circle (  1.96);

\path[draw=drawColor,line width= 0.6pt,line join=round] (213.07, 33.43) -- (213.07, 40.74);

\path[draw=drawColor,line width= 0.6pt,line join=round] (213.07, 28.24) -- (213.07, 27.04);
\definecolor{fillColor}{RGB}{255,255,255}

\path[draw=drawColor,line width= 0.6pt,fill=fillColor] (187.27, 33.43) --
	(187.27, 28.24) --
	(238.88, 28.24) --
	(238.88, 33.43) --
	(187.27, 33.43) --
	cycle;

\path[draw=drawColor,line width= 1.1pt] (187.27, 29.91) -- (238.88, 29.91);
\definecolor{fillColor}{gray}{0.20}

\path[draw=drawColor,line width= 0.4pt,line join=round,line cap=round,fill=fillColor] (281.88, 41.69) circle (  1.96);

\path[draw=drawColor,line width= 0.4pt,line join=round,line cap=round,fill=fillColor] (281.88, 46.33) circle (  1.96);

\path[draw=drawColor,line width= 0.4pt,line join=round,line cap=round,fill=fillColor] (281.88, 42.05) circle (  1.96);

\path[draw=drawColor,line width= 0.4pt,line join=round,line cap=round,fill=fillColor] (281.88, 45.45) circle (  1.96);

\path[draw=drawColor,line width= 0.4pt,line join=round,line cap=round,fill=fillColor] (281.88, 42.12) circle (  1.96);

\path[draw=drawColor,line width= 0.4pt,line join=round,line cap=round,fill=fillColor] (281.88, 56.89) circle (  1.96);

\path[draw=drawColor,line width= 0.4pt,line join=round,line cap=round,fill=fillColor] (281.88, 46.99) circle (  1.96);

\path[draw=drawColor,line width= 0.4pt,line join=round,line cap=round,fill=fillColor] (281.88, 51.73) circle (  1.96);

\path[draw=drawColor,line width= 0.4pt,line join=round,line cap=round,fill=fillColor] (281.88, 77.22) circle (  1.96);

\path[draw=drawColor,line width= 0.4pt,line join=round,line cap=round,fill=fillColor] (281.88, 62.61) circle (  1.96);

\path[draw=drawColor,line width= 0.4pt,line join=round,line cap=round,fill=fillColor] (281.88, 45.22) circle (  1.96);

\path[draw=drawColor,line width= 0.4pt,line join=round,line cap=round,fill=fillColor] (281.88, 43.58) circle (  1.96);

\path[draw=drawColor,line width= 0.4pt,line join=round,line cap=round,fill=fillColor] (281.88, 52.50) circle (  1.96);

\path[draw=drawColor,line width= 0.4pt,line join=round,line cap=round,fill=fillColor] (281.88, 43.03) circle (  1.96);

\path[draw=drawColor,line width= 0.4pt,line join=round,line cap=round,fill=fillColor] (281.88, 44.42) circle (  1.96);

\path[draw=drawColor,line width= 0.4pt,line join=round,line cap=round,fill=fillColor] (281.88, 48.01) circle (  1.96);

\path[draw=drawColor,line width= 0.4pt,line join=round,line cap=round,fill=fillColor] (281.88, 42.33) circle (  1.96);

\path[draw=drawColor,line width= 0.4pt,line join=round,line cap=round,fill=fillColor] (281.88, 53.66) circle (  1.96);

\path[draw=drawColor,line width= 0.4pt,line join=round,line cap=round,fill=fillColor] (281.88, 52.65) circle (  1.96);

\path[draw=drawColor,line width= 0.4pt,line join=round,line cap=round,fill=fillColor] (281.88, 65.31) circle (  1.96);

\path[draw=drawColor,line width= 0.4pt,line join=round,line cap=round,fill=fillColor] (281.88, 42.32) circle (  1.96);

\path[draw=drawColor,line width= 0.4pt,line join=round,line cap=round,fill=fillColor] (281.88, 74.45) circle (  1.96);

\path[draw=drawColor,line width= 0.4pt,line join=round,line cap=round,fill=fillColor] (281.88, 43.15) circle (  1.96);

\path[draw=drawColor,line width= 0.4pt,line join=round,line cap=round,fill=fillColor] (281.88, 52.09) circle (  1.96);

\path[draw=drawColor,line width= 0.4pt,line join=round,line cap=round,fill=fillColor] (281.88, 46.31) circle (  1.96);

\path[draw=drawColor,line width= 0.4pt,line join=round,line cap=round,fill=fillColor] (281.88, 43.77) circle (  1.96);

\path[draw=drawColor,line width= 0.4pt,line join=round,line cap=round,fill=fillColor] (281.88, 42.39) circle (  1.96);

\path[draw=drawColor,line width= 0.4pt,line join=round,line cap=round,fill=fillColor] (281.88, 45.91) circle (  1.96);

\path[draw=drawColor,line width= 0.4pt,line join=round,line cap=round,fill=fillColor] (281.88, 70.36) circle (  1.96);

\path[draw=drawColor,line width= 0.4pt,line join=round,line cap=round,fill=fillColor] (281.88, 44.90) circle (  1.96);

\path[draw=drawColor,line width= 0.6pt,line join=round] (281.88, 33.45) -- (281.88, 41.49);

\path[draw=drawColor,line width= 0.6pt,line join=round] (281.88, 28.07) -- (281.88, 27.00);
\definecolor{fillColor}{RGB}{255,255,255}

\path[draw=drawColor,line width= 0.6pt,fill=fillColor] (256.08, 33.45) --
	(256.08, 28.07) --
	(307.69, 28.07) --
	(307.69, 33.45) --
	(256.08, 33.45) --
	cycle;

\path[draw=drawColor,line width= 1.1pt] (256.08, 29.75) -- (307.69, 29.75);
\definecolor{fillColor}{gray}{0.20}

\path[draw=drawColor,line width= 0.4pt,line join=round,line cap=round,fill=fillColor] (350.70, 53.38) circle (  1.96);

\path[draw=drawColor,line width= 0.4pt,line join=round,line cap=round,fill=fillColor] (350.70, 51.92) circle (  1.96);

\path[draw=drawColor,line width= 0.6pt,line join=round] (350.70, 45.78) -- (350.70, 50.70);

\path[draw=drawColor,line width= 0.6pt,line join=round] (350.70, 41.91) -- (350.70, 38.10);
\definecolor{fillColor}{RGB}{255,255,255}

\path[draw=drawColor,line width= 0.6pt,fill=fillColor] (324.89, 45.78) --
	(324.89, 41.91) --
	(376.50, 41.91) --
	(376.50, 45.78) --
	(324.89, 45.78) --
	cycle;

\path[draw=drawColor,line width= 1.1pt] (324.89, 43.79) -- (376.50, 43.79);
\end{scope}
\begin{scope}
\path[clip] (  0.00,  0.00) rectangle (397.48,216.81);
\definecolor{drawColor}{gray}{0.30}

\node[text=drawColor,anchor=base east,inner sep=0pt, outer sep=0pt, scale=  0.88] at ( 29.21, 23.51) {0.0};

\node[text=drawColor,anchor=base east,inner sep=0pt, outer sep=0pt, scale=  0.88] at ( 29.21, 65.25) {0.5};

\node[text=drawColor,anchor=base east,inner sep=0pt, outer sep=0pt, scale=  0.88] at ( 29.21,106.99) {1.0};

\node[text=drawColor,anchor=base east,inner sep=0pt, outer sep=0pt, scale=  0.88] at ( 29.21,148.73) {1.5};

\node[text=drawColor,anchor=base east,inner sep=0pt, outer sep=0pt, scale=  0.88] at ( 29.21,190.47) {2.0};
\end{scope}
\begin{scope}
\path[clip] (  0.00,  0.00) rectangle (397.48,216.81);
\definecolor{drawColor}{gray}{0.20}

\path[draw=drawColor,line width= 0.6pt,line join=round] ( 31.41, 26.54) --
	( 34.16, 26.54);

\path[draw=drawColor,line width= 0.6pt,line join=round] ( 31.41, 68.28) --
	( 34.16, 68.28);

\path[draw=drawColor,line width= 0.6pt,line join=round] ( 31.41,110.02) --
	( 34.16,110.02);

\path[draw=drawColor,line width= 0.6pt,line join=round] ( 31.41,151.76) --
	( 34.16,151.76);

\path[draw=drawColor,line width= 0.6pt,line join=round] ( 31.41,193.50) --
	( 34.16,193.50);
\end{scope}
\begin{scope}
\path[clip] (  0.00,  0.00) rectangle (397.48,216.81);
\definecolor{drawColor}{gray}{0.20}

\path[draw=drawColor,line width= 0.6pt,line join=round] ( 75.44, 15.47) --
	( 75.44, 18.22);

\path[draw=drawColor,line width= 0.6pt,line join=round] (144.26, 15.47) --
	(144.26, 18.22);

\path[draw=drawColor,line width= 0.6pt,line join=round] (213.07, 15.47) --
	(213.07, 18.22);

\path[draw=drawColor,line width= 0.6pt,line join=round] (281.88, 15.47) --
	(281.88, 18.22);

\path[draw=drawColor,line width= 0.6pt,line join=round] (350.70, 15.47) --
	(350.70, 18.22);
\end{scope}
\begin{scope}
\path[clip] (  0.00,  0.00) rectangle (397.48,216.81);
\definecolor{drawColor}{gray}{0.30}

\node[text=drawColor,anchor=base,inner sep=0pt, outer sep=0pt, scale=  0.88] at ( 75.44,  7.21) {$\lambda=0.05$};

\node[text=drawColor,anchor=base,inner sep=0pt, outer sep=0pt, scale=  0.88] at (144.26,  7.21) {$\lambda=0.01$};

\node[text=drawColor,anchor=base,inner sep=0pt, outer sep=0pt, scale=  0.88] at (213.07,  7.21) {$\lambda=0.001$};

\node[text=drawColor,anchor=base,inner sep=0pt, outer sep=0pt, scale=  0.88] at (281.88,  7.21) {$\lambda=0$};

\node[text=drawColor,anchor=base,inner sep=0pt, outer sep=0pt, scale=  0.88] at (350.70,  7.21) {Monte Carlo};
\end{scope}
\begin{scope}
\path[clip] (  0.00,  0.00) rectangle (397.48,216.81);
\definecolor{drawColor}{RGB}{0,0,0}

\node[text=drawColor,rotate= 90.00,anchor=base,inner sep=0pt, outer sep=0pt, scale=  1.10] at ( 13.08,114.77) {mean squared error};
\end{scope}
\end{tikzpicture}

%% file: boxplotdifferentlambdafeaturemaps.tex
\begin{tikzpicture}[x=1pt,y=1pt]
\definecolor{fillColor}{RGB}{255,255,255}
\path[use as bounding box,fill=fillColor,fill opacity=0.00] (0,0) rectangle (397.48,216.81);
\begin{scope}
\path[clip] (  0.00,  0.00) rectangle (397.48,216.81);
\definecolor{drawColor}{RGB}{255,255,255}
\definecolor{fillColor}{RGB}{255,255,255}

\path[draw=drawColor,line width= 0.6pt,line join=round,line cap=round,fill=fillColor] (  0.00,  0.00) rectangle (397.48,216.81);
\end{scope}
\begin{scope}
\path[clip] ( 34.16, 18.22) rectangle (391.98,211.31);
\definecolor{fillColor}{gray}{0.92}

\path[fill=fillColor] ( 34.16, 18.22) rectangle (391.98,211.31);
\definecolor{drawColor}{RGB}{255,255,255}

\path[draw=drawColor,line width= 0.3pt,line join=round] ( 34.16, 51.94) --
	(391.98, 51.94);

\path[draw=drawColor,line width= 0.3pt,line join=round] ( 34.16,102.24) --
	(391.98,102.24);

\path[draw=drawColor,line width= 0.3pt,line join=round] ( 34.16,152.55) --
	(391.98,152.55);

\path[draw=drawColor,line width= 0.3pt,line join=round] ( 34.16,202.85) --
	(391.98,202.85);

\path[draw=drawColor,line width= 0.6pt,line join=round] ( 34.16, 26.79) --
	(391.98, 26.79);

\path[draw=drawColor,line width= 0.6pt,line join=round] ( 34.16, 77.09) --
	(391.98, 77.09);

\path[draw=drawColor,line width= 0.6pt,line join=round] ( 34.16,127.39) --
	(391.98,127.39);

\path[draw=drawColor,line width= 0.6pt,line join=round] ( 34.16,177.70) --
	(391.98,177.70);

\path[draw=drawColor,line width= 0.6pt,line join=round] ( 68.78, 18.22) --
	( 68.78,211.31);

\path[draw=drawColor,line width= 0.6pt,line join=round] (126.50, 18.22) --
	(126.50,211.31);

\path[draw=drawColor,line width= 0.6pt,line join=round] (184.21, 18.22) --
	(184.21,211.31);

\path[draw=drawColor,line width= 0.6pt,line join=round] (241.93, 18.22) --
	(241.93,211.31);

\path[draw=drawColor,line width= 0.6pt,line join=round] (299.64, 18.22) --
	(299.64,211.31);

\path[draw=drawColor,line width= 0.6pt,line join=round] (357.36, 18.22) --
	(357.36,211.31);
\definecolor{drawColor}{gray}{0.20}
\definecolor{fillColor}{gray}{0.20}

\path[draw=drawColor,line width= 0.4pt,line join=round,line cap=round,fill=fillColor] ( 68.78,131.85) circle (  1.96);

\path[draw=drawColor,line width= 0.4pt,line join=round,line cap=round,fill=fillColor] ( 68.78,101.23) circle (  1.96);

\path[draw=drawColor,line width= 0.4pt,line join=round,line cap=round,fill=fillColor] ( 68.78,101.96) circle (  1.96);

\path[draw=drawColor,line width= 0.4pt,line join=round,line cap=round,fill=fillColor] ( 68.78,159.54) circle (  1.96);

\path[draw=drawColor,line width= 0.4pt,line join=round,line cap=round,fill=fillColor] ( 68.78,101.18) circle (  1.96);

\path[draw=drawColor,line width= 0.4pt,line join=round,line cap=round,fill=fillColor] ( 68.78,124.94) circle (  1.96);

\path[draw=drawColor,line width= 0.4pt,line join=round,line cap=round,fill=fillColor] ( 68.78,110.00) circle (  1.96);

\path[draw=drawColor,line width= 0.4pt,line join=round,line cap=round,fill=fillColor] ( 68.78,106.13) circle (  1.96);

\path[draw=drawColor,line width= 0.4pt,line join=round,line cap=round,fill=fillColor] ( 68.78,106.22) circle (  1.96);

\path[draw=drawColor,line width= 0.4pt,line join=round,line cap=round,fill=fillColor] ( 68.78,107.01) circle (  1.96);

\path[draw=drawColor,line width= 0.4pt,line join=round,line cap=round,fill=fillColor] ( 68.78,135.51) circle (  1.96);

\path[draw=drawColor,line width= 0.4pt,line join=round,line cap=round,fill=fillColor] ( 68.78,109.47) circle (  1.96);

\path[draw=drawColor,line width= 0.4pt,line join=round,line cap=round,fill=fillColor] ( 68.78,107.60) circle (  1.96);

\path[draw=drawColor,line width= 0.4pt,line join=round,line cap=round,fill=fillColor] ( 68.78,117.57) circle (  1.96);

\path[draw=drawColor,line width= 0.4pt,line join=round,line cap=round,fill=fillColor] ( 68.78,117.64) circle (  1.96);

\path[draw=drawColor,line width= 0.4pt,line join=round,line cap=round,fill=fillColor] ( 68.78,102.76) circle (  1.96);

\path[draw=drawColor,line width= 0.4pt,line join=round,line cap=round,fill=fillColor] ( 68.78,101.31) circle (  1.96);

\path[draw=drawColor,line width= 0.4pt,line join=round,line cap=round,fill=fillColor] ( 68.78,104.59) circle (  1.96);

\path[draw=drawColor,line width= 0.4pt,line join=round,line cap=round,fill=fillColor] ( 68.78,117.17) circle (  1.96);

\path[draw=drawColor,line width= 0.4pt,line join=round,line cap=round,fill=fillColor] ( 68.78,102.08) circle (  1.96);

\path[draw=drawColor,line width= 0.4pt,line join=round,line cap=round,fill=fillColor] ( 68.78,115.39) circle (  1.96);

\path[draw=drawColor,line width= 0.4pt,line join=round,line cap=round,fill=fillColor] ( 68.78,102.30) circle (  1.96);

\path[draw=drawColor,line width= 0.4pt,line join=round,line cap=round,fill=fillColor] ( 68.78,109.84) circle (  1.96);

\path[draw=drawColor,line width= 0.4pt,line join=round,line cap=round,fill=fillColor] ( 68.78,108.36) circle (  1.96);

\path[draw=drawColor,line width= 0.4pt,line join=round,line cap=round,fill=fillColor] ( 68.78,117.13) circle (  1.96);

\path[draw=drawColor,line width= 0.4pt,line join=round,line cap=round,fill=fillColor] ( 68.78,128.77) circle (  1.96);

\path[draw=drawColor,line width= 0.4pt,line join=round,line cap=round,fill=fillColor] ( 68.78,109.57) circle (  1.96);

\path[draw=drawColor,line width= 0.4pt,line join=round,line cap=round,fill=fillColor] ( 68.78,108.32) circle (  1.96);

\path[draw=drawColor,line width= 0.4pt,line join=round,line cap=round,fill=fillColor] ( 68.78,142.96) circle (  1.96);

\path[draw=drawColor,line width= 0.4pt,line join=round,line cap=round,fill=fillColor] ( 68.78,100.96) circle (  1.96);

\path[draw=drawColor,line width= 0.6pt,line join=round] ( 68.78, 78.23) -- ( 68.78,100.25);

\path[draw=drawColor,line width= 0.6pt,line join=round] ( 68.78, 63.33) -- ( 68.78, 53.21);
\definecolor{fillColor}{RGB}{255,255,255}

\path[draw=drawColor,line width= 0.6pt,fill=fillColor] ( 47.14, 78.23) --
	( 47.14, 63.33) --
	( 90.43, 63.33) --
	( 90.43, 78.23) --
	( 47.14, 78.23) --
	cycle;

\path[draw=drawColor,line width= 1.1pt] ( 47.14, 69.46) -- ( 90.43, 69.46);
\definecolor{fillColor}{gray}{0.20}

\path[draw=drawColor,line width= 0.4pt,line join=round,line cap=round,fill=fillColor] (126.50,100.49) circle (  1.96);

\path[draw=drawColor,line width= 0.4pt,line join=round,line cap=round,fill=fillColor] (126.50,152.52) circle (  1.96);

\path[draw=drawColor,line width= 0.4pt,line join=round,line cap=round,fill=fillColor] (126.50,120.74) circle (  1.96);

\path[draw=drawColor,line width= 0.4pt,line join=round,line cap=round,fill=fillColor] (126.50,101.51) circle (  1.96);

\path[draw=drawColor,line width= 0.4pt,line join=round,line cap=round,fill=fillColor] (126.50, 81.01) circle (  1.96);

\path[draw=drawColor,line width= 0.4pt,line join=round,line cap=round,fill=fillColor] (126.50, 74.97) circle (  1.96);

\path[draw=drawColor,line width= 0.4pt,line join=round,line cap=round,fill=fillColor] (126.50,101.44) circle (  1.96);

\path[draw=drawColor,line width= 0.4pt,line join=round,line cap=round,fill=fillColor] (126.50, 75.95) circle (  1.96);

\path[draw=drawColor,line width= 0.4pt,line join=round,line cap=round,fill=fillColor] (126.50,116.46) circle (  1.96);

\path[draw=drawColor,line width= 0.4pt,line join=round,line cap=round,fill=fillColor] (126.50, 74.71) circle (  1.96);

\path[draw=drawColor,line width= 0.4pt,line join=round,line cap=round,fill=fillColor] (126.50, 81.37) circle (  1.96);

\path[draw=drawColor,line width= 0.4pt,line join=round,line cap=round,fill=fillColor] (126.50, 81.87) circle (  1.96);

\path[draw=drawColor,line width= 0.4pt,line join=round,line cap=round,fill=fillColor] (126.50, 93.21) circle (  1.96);

\path[draw=drawColor,line width= 0.4pt,line join=round,line cap=round,fill=fillColor] (126.50, 76.26) circle (  1.96);

\path[draw=drawColor,line width= 0.4pt,line join=round,line cap=round,fill=fillColor] (126.50, 77.55) circle (  1.96);

\path[draw=drawColor,line width= 0.4pt,line join=round,line cap=round,fill=fillColor] (126.50, 78.77) circle (  1.96);

\path[draw=drawColor,line width= 0.4pt,line join=round,line cap=round,fill=fillColor] (126.50,106.53) circle (  1.96);

\path[draw=drawColor,line width= 0.4pt,line join=round,line cap=round,fill=fillColor] (126.50, 91.80) circle (  1.96);

\path[draw=drawColor,line width= 0.4pt,line join=round,line cap=round,fill=fillColor] (126.50, 85.16) circle (  1.96);

\path[draw=drawColor,line width= 0.4pt,line join=round,line cap=round,fill=fillColor] (126.50, 86.58) circle (  1.96);

\path[draw=drawColor,line width= 0.4pt,line join=round,line cap=round,fill=fillColor] (126.50, 86.13) circle (  1.96);

\path[draw=drawColor,line width= 0.4pt,line join=round,line cap=round,fill=fillColor] (126.50, 74.42) circle (  1.96);

\path[draw=drawColor,line width= 0.4pt,line join=round,line cap=round,fill=fillColor] (126.50, 87.74) circle (  1.96);

\path[draw=drawColor,line width= 0.4pt,line join=round,line cap=round,fill=fillColor] (126.50,110.00) circle (  1.96);

\path[draw=drawColor,line width= 0.4pt,line join=round,line cap=round,fill=fillColor] (126.50,109.47) circle (  1.96);

\path[draw=drawColor,line width= 0.4pt,line join=round,line cap=round,fill=fillColor] (126.50, 76.92) circle (  1.96);

\path[draw=drawColor,line width= 0.4pt,line join=round,line cap=round,fill=fillColor] (126.50,162.68) circle (  1.96);

\path[draw=drawColor,line width= 0.4pt,line join=round,line cap=round,fill=fillColor] (126.50,113.12) circle (  1.96);

\path[draw=drawColor,line width= 0.4pt,line join=round,line cap=round,fill=fillColor] (126.50, 81.39) circle (  1.96);

\path[draw=drawColor,line width= 0.4pt,line join=round,line cap=round,fill=fillColor] (126.50, 80.62) circle (  1.96);

\path[draw=drawColor,line width= 0.4pt,line join=round,line cap=round,fill=fillColor] (126.50, 86.83) circle (  1.96);

\path[draw=drawColor,line width= 0.4pt,line join=round,line cap=round,fill=fillColor] (126.50, 96.63) circle (  1.96);

\path[draw=drawColor,line width= 0.4pt,line join=round,line cap=round,fill=fillColor] (126.50, 98.12) circle (  1.96);

\path[draw=drawColor,line width= 0.4pt,line join=round,line cap=round,fill=fillColor] (126.50, 84.55) circle (  1.96);

\path[draw=drawColor,line width= 0.4pt,line join=round,line cap=round,fill=fillColor] (126.50,123.73) circle (  1.96);

\path[draw=drawColor,line width= 0.6pt,line join=round] (126.50, 48.38) -- (126.50, 74.26);

\path[draw=drawColor,line width= 0.6pt,line join=round] (126.50, 31.10) -- (126.50, 28.00);
\definecolor{fillColor}{RGB}{255,255,255}

\path[draw=drawColor,line width= 0.6pt,fill=fillColor] (104.86, 48.38) --
	(104.86, 31.10) --
	(148.14, 31.10) --
	(148.14, 48.38) --
	(104.86, 48.38) --
	cycle;

\path[draw=drawColor,line width= 1.1pt] (104.86, 36.38) -- (148.14, 36.38);
\definecolor{fillColor}{gray}{0.20}

\path[draw=drawColor,line width= 0.4pt,line join=round,line cap=round,fill=fillColor] (184.21,103.41) circle (  1.96);

\path[draw=drawColor,line width= 0.4pt,line join=round,line cap=round,fill=fillColor] (184.21,164.01) circle (  1.96);

\path[draw=drawColor,line width= 0.4pt,line join=round,line cap=round,fill=fillColor] (184.21,139.76) circle (  1.96);

\path[draw=drawColor,line width= 0.4pt,line join=round,line cap=round,fill=fillColor] (184.21,100.58) circle (  1.96);

\path[draw=drawColor,line width= 0.4pt,line join=round,line cap=round,fill=fillColor] (184.21,104.51) circle (  1.96);

\path[draw=drawColor,line width= 0.4pt,line join=round,line cap=round,fill=fillColor] (184.21,121.14) circle (  1.96);

\path[draw=drawColor,line width= 0.4pt,line join=round,line cap=round,fill=fillColor] (184.21, 93.71) circle (  1.96);

\path[draw=drawColor,line width= 0.4pt,line join=round,line cap=round,fill=fillColor] (184.21, 82.33) circle (  1.96);

\path[draw=drawColor,line width= 0.4pt,line join=round,line cap=round,fill=fillColor] (184.21, 83.34) circle (  1.96);

\path[draw=drawColor,line width= 0.4pt,line join=round,line cap=round,fill=fillColor] (184.21, 96.78) circle (  1.96);

\path[draw=drawColor,line width= 0.4pt,line join=round,line cap=round,fill=fillColor] (184.21, 83.04) circle (  1.96);

\path[draw=drawColor,line width= 0.4pt,line join=round,line cap=round,fill=fillColor] (184.21,119.06) circle (  1.96);

\path[draw=drawColor,line width= 0.4pt,line join=round,line cap=round,fill=fillColor] (184.21, 97.66) circle (  1.96);

\path[draw=drawColor,line width= 0.4pt,line join=round,line cap=round,fill=fillColor] (184.21, 90.43) circle (  1.96);

\path[draw=drawColor,line width= 0.4pt,line join=round,line cap=round,fill=fillColor] (184.21, 94.70) circle (  1.96);

\path[draw=drawColor,line width= 0.4pt,line join=round,line cap=round,fill=fillColor] (184.21, 91.77) circle (  1.96);

\path[draw=drawColor,line width= 0.4pt,line join=round,line cap=round,fill=fillColor] (184.21, 97.12) circle (  1.96);

\path[draw=drawColor,line width= 0.4pt,line join=round,line cap=round,fill=fillColor] (184.21,120.77) circle (  1.96);

\path[draw=drawColor,line width= 0.4pt,line join=round,line cap=round,fill=fillColor] (184.21,110.54) circle (  1.96);

\path[draw=drawColor,line width= 0.4pt,line join=round,line cap=round,fill=fillColor] (184.21, 80.44) circle (  1.96);

\path[draw=drawColor,line width= 0.4pt,line join=round,line cap=round,fill=fillColor] (184.21,191.00) circle (  1.96);

\path[draw=drawColor,line width= 0.4pt,line join=round,line cap=round,fill=fillColor] (184.21,130.29) circle (  1.96);

\path[draw=drawColor,line width= 0.4pt,line join=round,line cap=round,fill=fillColor] (184.21, 87.85) circle (  1.96);

\path[draw=drawColor,line width= 0.4pt,line join=round,line cap=round,fill=fillColor] (184.21, 83.78) circle (  1.96);

\path[draw=drawColor,line width= 0.4pt,line join=round,line cap=round,fill=fillColor] (184.21, 94.21) circle (  1.96);

\path[draw=drawColor,line width= 0.4pt,line join=round,line cap=round,fill=fillColor] (184.21,106.63) circle (  1.96);

\path[draw=drawColor,line width= 0.4pt,line join=round,line cap=round,fill=fillColor] (184.21,110.36) circle (  1.96);

\path[draw=drawColor,line width= 0.4pt,line join=round,line cap=round,fill=fillColor] (184.21, 89.51) circle (  1.96);

\path[draw=drawColor,line width= 0.4pt,line join=round,line cap=round,fill=fillColor] (184.21,126.64) circle (  1.96);

\path[draw=drawColor,line width= 0.6pt,line join=round] (184.21, 50.04) -- (184.21, 80.13);

\path[draw=drawColor,line width= 0.6pt,line join=round] (184.21, 29.89) -- (184.21, 27.10);
\definecolor{fillColor}{RGB}{255,255,255}

\path[draw=drawColor,line width= 0.6pt,fill=fillColor] (162.57, 50.04) --
	(162.57, 29.89) --
	(205.86, 29.89) --
	(205.86, 50.04) --
	(162.57, 50.04) --
	cycle;

\path[draw=drawColor,line width= 1.1pt] (162.57, 36.08) -- (205.86, 36.08);
\definecolor{fillColor}{gray}{0.20}

\path[draw=drawColor,line width= 0.4pt,line join=round,line cap=round,fill=fillColor] (241.93,104.59) circle (  1.96);

\path[draw=drawColor,line width= 0.4pt,line join=round,line cap=round,fill=fillColor] (241.93,168.38) circle (  1.96);

\path[draw=drawColor,line width= 0.4pt,line join=round,line cap=round,fill=fillColor] (241.93, 81.19) circle (  1.96);

\path[draw=drawColor,line width= 0.4pt,line join=round,line cap=round,fill=fillColor] (241.93,146.89) circle (  1.96);

\path[draw=drawColor,line width= 0.4pt,line join=round,line cap=round,fill=fillColor] (241.93,100.62) circle (  1.96);

\path[draw=drawColor,line width= 0.4pt,line join=round,line cap=round,fill=fillColor] (241.93,105.37) circle (  1.96);

\path[draw=drawColor,line width= 0.4pt,line join=round,line cap=round,fill=fillColor] (241.93, 81.64) circle (  1.96);

\path[draw=drawColor,line width= 0.4pt,line join=round,line cap=round,fill=fillColor] (241.93,122.67) circle (  1.96);

\path[draw=drawColor,line width= 0.4pt,line join=round,line cap=round,fill=fillColor] (241.93, 98.38) circle (  1.96);

\path[draw=drawColor,line width= 0.4pt,line join=round,line cap=round,fill=fillColor] (241.93, 85.58) circle (  1.96);

\path[draw=drawColor,line width= 0.4pt,line join=round,line cap=round,fill=fillColor] (241.93, 83.85) circle (  1.96);

\path[draw=drawColor,line width= 0.4pt,line join=round,line cap=round,fill=fillColor] (241.93, 98.30) circle (  1.96);

\path[draw=drawColor,line width= 0.4pt,line join=round,line cap=round,fill=fillColor] (241.93, 84.83) circle (  1.96);

\path[draw=drawColor,line width= 0.4pt,line join=round,line cap=round,fill=fillColor] (241.93,122.82) circle (  1.96);

\path[draw=drawColor,line width= 0.4pt,line join=round,line cap=round,fill=fillColor] (241.93, 81.28) circle (  1.96);

\path[draw=drawColor,line width= 0.4pt,line join=round,line cap=round,fill=fillColor] (241.93,100.37) circle (  1.96);

\path[draw=drawColor,line width= 0.4pt,line join=round,line cap=round,fill=fillColor] (241.93, 92.41) circle (  1.96);

\path[draw=drawColor,line width= 0.4pt,line join=round,line cap=round,fill=fillColor] (241.93, 97.81) circle (  1.96);

\path[draw=drawColor,line width= 0.4pt,line join=round,line cap=round,fill=fillColor] (241.93, 94.18) circle (  1.96);

\path[draw=drawColor,line width= 0.4pt,line join=round,line cap=round,fill=fillColor] (241.93,100.73) circle (  1.96);

\path[draw=drawColor,line width= 0.4pt,line join=round,line cap=round,fill=fillColor] (241.93,124.67) circle (  1.96);

\path[draw=drawColor,line width= 0.4pt,line join=round,line cap=round,fill=fillColor] (241.93,110.73) circle (  1.96);

\path[draw=drawColor,line width= 0.4pt,line join=round,line cap=round,fill=fillColor] (241.93, 82.13) circle (  1.96);

\path[draw=drawColor,line width= 0.4pt,line join=round,line cap=round,fill=fillColor] (241.93,201.26) circle (  1.96);

\path[draw=drawColor,line width= 0.4pt,line join=round,line cap=round,fill=fillColor] (241.93,136.40) circle (  1.96);

\path[draw=drawColor,line width= 0.4pt,line join=round,line cap=round,fill=fillColor] (241.93, 90.57) circle (  1.96);

\path[draw=drawColor,line width= 0.4pt,line join=round,line cap=round,fill=fillColor] (241.93, 85.45) circle (  1.96);

\path[draw=drawColor,line width= 0.4pt,line join=round,line cap=round,fill=fillColor] (241.93, 96.85) circle (  1.96);

\path[draw=drawColor,line width= 0.4pt,line join=round,line cap=round,fill=fillColor] (241.93,110.18) circle (  1.96);

\path[draw=drawColor,line width= 0.4pt,line join=round,line cap=round,fill=fillColor] (241.93,114.62) circle (  1.96);

\path[draw=drawColor,line width= 0.4pt,line join=round,line cap=round,fill=fillColor] (241.93, 91.30) circle (  1.96);

\path[draw=drawColor,line width= 0.4pt,line join=round,line cap=round,fill=fillColor] (241.93,127.47) circle (  1.96);

\path[draw=drawColor,line width= 0.6pt,line join=round] (241.93, 50.09) -- (241.93, 80.17);

\path[draw=drawColor,line width= 0.6pt,line join=round] (241.93, 29.80) -- (241.93, 27.01);
\definecolor{fillColor}{RGB}{255,255,255}

\path[draw=drawColor,line width= 0.6pt,fill=fillColor] (220.28, 50.09) --
	(220.28, 29.80) --
	(263.57, 29.80) --
	(263.57, 50.09) --
	(220.28, 50.09) --
	cycle;

\path[draw=drawColor,line width= 1.1pt] (220.28, 36.12) -- (263.57, 36.12);
\definecolor{fillColor}{gray}{0.20}

\path[draw=drawColor,line width= 0.4pt,line join=round,line cap=round,fill=fillColor] (299.64,104.74) circle (  1.96);

\path[draw=drawColor,line width= 0.4pt,line join=round,line cap=round,fill=fillColor] (299.64,168.93) circle (  1.96);

\path[draw=drawColor,line width= 0.4pt,line join=round,line cap=round,fill=fillColor] (299.64, 81.44) circle (  1.96);

\path[draw=drawColor,line width= 0.4pt,line join=round,line cap=round,fill=fillColor] (299.64,147.80) circle (  1.96);

\path[draw=drawColor,line width= 0.4pt,line join=round,line cap=round,fill=fillColor] (299.64, 80.68) circle (  1.96);

\path[draw=drawColor,line width= 0.4pt,line join=round,line cap=round,fill=fillColor] (299.64,100.64) circle (  1.96);

\path[draw=drawColor,line width= 0.4pt,line join=round,line cap=round,fill=fillColor] (299.64,105.47) circle (  1.96);

\path[draw=drawColor,line width= 0.4pt,line join=round,line cap=round,fill=fillColor] (299.64, 81.81) circle (  1.96);

\path[draw=drawColor,line width= 0.4pt,line join=round,line cap=round,fill=fillColor] (299.64,122.85) circle (  1.96);

\path[draw=drawColor,line width= 0.4pt,line join=round,line cap=round,fill=fillColor] (299.64, 98.98) circle (  1.96);

\path[draw=drawColor,line width= 0.4pt,line join=round,line cap=round,fill=fillColor] (299.64, 86.01) circle (  1.96);

\path[draw=drawColor,line width= 0.4pt,line join=round,line cap=round,fill=fillColor] (299.64, 83.91) circle (  1.96);

\path[draw=drawColor,line width= 0.4pt,line join=round,line cap=round,fill=fillColor] (299.64, 98.49) circle (  1.96);

\path[draw=drawColor,line width= 0.4pt,line join=round,line cap=round,fill=fillColor] (299.64, 85.00) circle (  1.96);

\path[draw=drawColor,line width= 0.4pt,line join=round,line cap=round,fill=fillColor] (299.64,123.27) circle (  1.96);

\path[draw=drawColor,line width= 0.4pt,line join=round,line cap=round,fill=fillColor] (299.64, 81.58) circle (  1.96);

\path[draw=drawColor,line width= 0.4pt,line join=round,line cap=round,fill=fillColor] (299.64,100.71) circle (  1.96);

\path[draw=drawColor,line width= 0.4pt,line join=round,line cap=round,fill=fillColor] (299.64, 92.66) circle (  1.96);

\path[draw=drawColor,line width= 0.4pt,line join=round,line cap=round,fill=fillColor] (299.64, 98.20) circle (  1.96);

\path[draw=drawColor,line width= 0.4pt,line join=round,line cap=round,fill=fillColor] (299.64, 94.49) circle (  1.96);

\path[draw=drawColor,line width= 0.4pt,line join=round,line cap=round,fill=fillColor] (299.64,101.18) circle (  1.96);

\path[draw=drawColor,line width= 0.4pt,line join=round,line cap=round,fill=fillColor] (299.64,125.14) circle (  1.96);

\path[draw=drawColor,line width= 0.4pt,line join=round,line cap=round,fill=fillColor] (299.64,110.74) circle (  1.96);

\path[draw=drawColor,line width= 0.4pt,line join=round,line cap=round,fill=fillColor] (299.64, 82.34) circle (  1.96);

\path[draw=drawColor,line width= 0.4pt,line join=round,line cap=round,fill=fillColor] (299.64,202.53) circle (  1.96);

\path[draw=drawColor,line width= 0.4pt,line join=round,line cap=round,fill=fillColor] (299.64,137.15) circle (  1.96);

\path[draw=drawColor,line width= 0.4pt,line join=round,line cap=round,fill=fillColor] (299.64, 90.91) circle (  1.96);

\path[draw=drawColor,line width= 0.4pt,line join=round,line cap=round,fill=fillColor] (299.64, 85.66) circle (  1.96);

\path[draw=drawColor,line width= 0.4pt,line join=round,line cap=round,fill=fillColor] (299.64, 97.16) circle (  1.96);

\path[draw=drawColor,line width= 0.4pt,line join=round,line cap=round,fill=fillColor] (299.64,110.60) circle (  1.96);

\path[draw=drawColor,line width= 0.4pt,line join=round,line cap=round,fill=fillColor] (299.64,115.15) circle (  1.96);

\path[draw=drawColor,line width= 0.4pt,line join=round,line cap=round,fill=fillColor] (299.64, 91.53) circle (  1.96);

\path[draw=drawColor,line width= 0.4pt,line join=round,line cap=round,fill=fillColor] (299.64,127.57) circle (  1.96);

\path[draw=drawColor,line width= 0.6pt,line join=round] (299.64, 50.10) -- (299.64, 79.99);

\path[draw=drawColor,line width= 0.6pt,line join=round] (299.64, 29.75) -- (299.64, 27.00);
\definecolor{fillColor}{RGB}{255,255,255}

\path[draw=drawColor,line width= 0.6pt,fill=fillColor] (278.00, 50.10) --
	(278.00, 29.75) --
	(321.28, 29.75) --
	(321.28, 50.10) --
	(278.00, 50.10) --
	cycle;

\path[draw=drawColor,line width= 1.1pt] (278.00, 36.10) -- (321.28, 36.10);
\definecolor{fillColor}{gray}{0.20}

\path[draw=drawColor,line width= 0.4pt,line join=round,line cap=round,fill=fillColor] (357.36,107.65) circle (  1.96);

\path[draw=drawColor,line width= 0.4pt,line join=round,line cap=round,fill=fillColor] (357.36,103.25) circle (  1.96);

\path[draw=drawColor,line width= 0.6pt,line join=round] (357.36, 84.77) -- (357.36, 99.60);

\path[draw=drawColor,line width= 0.6pt,line join=round] (357.36, 73.12) -- (357.36, 61.64);
\definecolor{fillColor}{RGB}{255,255,255}

\path[draw=drawColor,line width= 0.6pt,fill=fillColor] (335.71, 84.77) --
	(335.71, 73.12) --
	(379.00, 73.12) --
	(379.00, 84.77) --
	(335.71, 84.77) --
	cycle;

\path[draw=drawColor,line width= 1.1pt] (335.71, 78.75) -- (379.00, 78.75);
\end{scope}
\begin{scope}
\path[clip] (  0.00,  0.00) rectangle (397.48,216.81);
\definecolor{drawColor}{gray}{0.30}

\node[text=drawColor,anchor=base east,inner sep=0pt, outer sep=0pt, scale=  0.88] at ( 29.21, 23.76) {0.0};

\node[text=drawColor,anchor=base east,inner sep=0pt, outer sep=0pt, scale=  0.88] at ( 29.21, 74.06) {0.2};

\node[text=drawColor,anchor=base east,inner sep=0pt, outer sep=0pt, scale=  0.88] at ( 29.21,124.36) {0.4};

\node[text=drawColor,anchor=base east,inner sep=0pt, outer sep=0pt, scale=  0.88] at ( 29.21,174.67) {0.6};
\end{scope}
\begin{scope}
\path[clip] (  0.00,  0.00) rectangle (397.48,216.81);
\definecolor{drawColor}{gray}{0.20}

\path[draw=drawColor,line width= 0.6pt,line join=round] ( 31.41, 26.79) --
	( 34.16, 26.79);

\path[draw=drawColor,line width= 0.6pt,line join=round] ( 31.41, 77.09) --
	( 34.16, 77.09);

\path[draw=drawColor,line width= 0.6pt,line join=round] ( 31.41,127.39) --
	( 34.16,127.39);

\path[draw=drawColor,line width= 0.6pt,line join=round] ( 31.41,177.70) --
	( 34.16,177.70);
\end{scope}
\begin{scope}
\path[clip] (  0.00,  0.00) rectangle (397.48,216.81);
\definecolor{drawColor}{gray}{0.20}

\path[draw=drawColor,line width= 0.6pt,line join=round] ( 68.78, 15.47) --
	( 68.78, 18.22);

\path[draw=drawColor,line width= 0.6pt,line join=round] (126.50, 15.47) --
	(126.50, 18.22);

\path[draw=drawColor,line width= 0.6pt,line join=round] (184.21, 15.47) --
	(184.21, 18.22);

\path[draw=drawColor,line width= 0.6pt,line join=round] (241.93, 15.47) --
	(241.93, 18.22);

\path[draw=drawColor,line width= 0.6pt,line join=round] (299.64, 15.47) --
	(299.64, 18.22);

\path[draw=drawColor,line width= 0.6pt,line join=round] (357.36, 15.47) --
	(357.36, 18.22);
\end{scope}
\begin{scope}
\path[clip] (  0.00,  0.00) rectangle (397.48,216.81);
\definecolor{drawColor}{gray}{0.30}

\node[text=drawColor,anchor=base,inner sep=0pt, outer sep=0pt, scale=  0.88] at ( 68.78,  7.21) {$\lambda=1$};

\node[text=drawColor,anchor=base,inner sep=0pt, outer sep=0pt, scale=  0.88] at (126.50,  7.21) {$\lambda=0.05$};

\node[text=drawColor,anchor=base,inner sep=0pt, outer sep=0pt, scale=  0.88] at (184.21,  7.21) {$\lambda=0.01$};

\node[text=drawColor,anchor=base,inner sep=0pt, outer sep=0pt, scale=  0.88] at (241.93,  7.21) {$\lambda=0.001$};

\node[text=drawColor,anchor=base,inner sep=0pt, outer sep=0pt, scale=  0.88] at (299.64,  7.21) {$\lambda=0$};

\node[text=drawColor,anchor=base,inner sep=0pt, outer sep=0pt, scale=  0.88] at (357.36,  7.21) {Monte Carlo};
\end{scope}
\begin{scope}
\path[clip] (  0.00,  0.00) rectangle (397.48,216.81);
\definecolor{drawColor}{RGB}{0,0,0}

\node[text=drawColor,rotate= 90.00,anchor=base,inner sep=0pt, outer sep=0pt, scale=  1.10] at ( 13.08,114.77) {mean squared error};
\end{scope}
\end{tikzpicture}

%% file: paper.bbl
\newcommand{\etalchar}[1]{$^{#1}$}
\begin{thebibliography}{BBG{\etalchar{+}}21}

\bibitem[AF03]{adams2003sobolev}
Robert~A. Adams and John J.~F. Fournier.
\newblock {\em Sobolev Spaces}.
\newblock Pure and Applied Mathematics. Academic Press, 2nd edition, 2003.

\bibitem[Aro50]{aronszajn1950theory}
Nachman Aronszajn.
\newblock {Theory of Reproducing Kernels}.
\newblock {\em Transactions of the American Mathematical Society},
  68(3):337--404, 1950.

\bibitem[BBG{\etalchar{+}}21]{Beck_2021}
Christian Beck, Sebastian Becker, Philipp Grohs, Nor Jaafari, and Arnulf
  Jentzen.
\newblock Solving the kolmogorov pde by means of deep learning.
\newblock {\em Journal of Scientific Computing}, 88(3), July 2021.

\bibitem[BC24]{Crepey2024}
Cyril Bénézet and Stéphane Crépey.
\newblock Handling model risk with xvas.
\newblock {\em Frontiers of Mathematical Finance}, 3:490--519, 01 2024.

\bibitem[BDG24]{BenthDeteringGalimbertiFlowForwards}
Fred~Espen Benth, Nils Detering, and Luca Galimberti.
\newblock Pricing options on flow forwards by neural networks in a hilbert
  space.
\newblock {\em Finance and Stochastics}, 28(1):81--121, 2024.

\bibitem[BS73]{BlackScholes1973}
Fischer Black and Myron Scholes.
\newblock {The Pricing of Options and Corporate Liabilities}.
\newblock {\em Journal of Political Economy}, 81(3):637--654, 1973.

\bibitem[Con04]{ContModelRisk}
Rama Cont.
\newblock Model uncertainty and its impact on the pricing of derivative
  instruments.
\newblock {\em Mathematical Finance}, 16, 07 2004.

\bibitem[DP16]{Detering2016}
Nils Detering and Natalie Packham.
\newblock Model risk of contingent claims.
\newblock {\em Quantitative Finance}, 16(9):1357--1374, 2016.

\bibitem[DP21]{dommel2021uniformfunctionestimatorsreproducing}
Paul Dommel and Alois Pichler.
\newblock Uniform function estimators in reproducing kernel hilbert spaces,
  2021.

\bibitem[DS88]{dunford1988linear_spectral_theory}
Neslon Dunford and Jakob~T.\ Schwartz.
\newblock {\em {Linear Operators, Part 2: Spectral Theory, Self Adjoint
  Operators in Hilbert Space}}.
\newblock Wiley Classics Library. Wiley, 1988.

\bibitem[DS98]{Draper1998}
Norman~R. Draper and Harry Smith.
\newblock {\em Applied Regression Analysis}.
\newblock John Wiley \& Sons, New York, NY, 3rd edition, 1998.

\bibitem[Fil01]{filipovic}
Damir Filipovi\'{c}.
\newblock {\em Consistency Problems for Heath--Jarrow--Morton Interest Rate
  Models}.
\newblock Springer Science \& Business Media, 2001.

\bibitem[FPP22]{RolfP-2025}
Magnus~Gr{\o}nnegaard Frandsen, Tobias~Cramer Pedersen, and Rolf Poulsen.
\newblock Delta force: option pricing with differential machine learning.
\newblock {\em Digital Finance}, 4(1):1--15, 2022.

\bibitem[FPY22]{filipovic2022stripping}
Damir Filipovi\'{c}, Markus Pelger, and Ye~Ye.
\newblock Stripping the discount curve --- a robust machine learning approach.
\newblock {\em Forthcoming, Management Science}, 2022.

\bibitem[FS25]{filipovic2025jointestimationconditionalmean}
Damir Filipovi\'{c} and Paul Schneider.
\newblock Joint estimation of conditional mean and covariance for unbalanced
  panels, 2025.

\bibitem[Gla04]{glasserman2004}
Paul Glasserman.
\newblock {\em Monte Carlo methods in financial engineering}.
\newblock Springer, New York, 2004.

\bibitem[H{\"a}r90]{Hardle1990}
Wolfgang H{\"a}rdle.
\newblock {\em Applied Nonparametric Regression}.
\newblock Cambridge University Press, Cambridge, UK, 1990.

\bibitem[HM85]{Hardle1985}
Wolfgang H{\"a}rdle and James~S. Marron.
\newblock Optimal bandwidth selection in nonparametric regression function
  estimation.
\newblock {\em The Annals of Statistics}, 13(4):1465--1481, 1985.

\bibitem[HS20]{HugeSavine2020}
Brian Huge and Antoine Savine.
\newblock Differential machine learning: The shape of things to come.
\newblock {\em Risk}, pages 76--81, 2020.

\bibitem[HSW89]{hornik1989}
Kurt Hornik, Maxwell Stinchcombe, and Halbert White.
\newblock Multilayer feedforward networks are universal approximators.
\newblock {\em Neural Networks}, 2(5):359--366, 1989.

\bibitem[{Jon}15]{manton2015rkhs}
{Jonathan H.Manton and Pierre-Olivier Amblard}.
\newblock {\em A Primer on Reproducing Kernel Hilbert Spaces}, volume~8.
\newblock Foundations and Trends$^\circledR$ in Signal Processing, 2015.

\bibitem[JS20]{JokhadzeSchmidt2020}
Valeriane Jokhadze and Wolfgang~M. Schmidt.
\newblock Measuring model risk in financial risk management and pricing.
\newblock {\em International Journal of Theoretical and Applied Finance},
  23(02):2050012, 2020.

\bibitem[LQT24]{LazarQiTunaru}
Emese Lazar, Shuyuan Qi, and Radu Tunaru.
\newblock Measures of model risk for continuous-time finance models.
\newblock {\em Journal of Financial Econometrics}, 22(5):1456--1481, 02 2024.

\bibitem[MPV12]{Montgomery2012}
Douglas~C. Montgomery, Elizabeth~A. Peck, and G.~Geoffrey Vining.
\newblock {\em Introduction to Linear Regression Analysis}.
\newblock John Wiley \& Sons, Hoboken, NJ, 5th edition, 2012.

\bibitem[Pag18]{pages2018}
Gilles Pagès.
\newblock {\em {Numerical Probability: An Introduction with Applications to
  Finance}}.
\newblock Universitext. Springer Cham, 2018.

\bibitem[Pal94]{palmer1994}
Theodore~W. Palmer.
\newblock {\em {Banach Algebras and the General Theory of *-Algebras}}.
\newblock Encyclopedia of Mathematics and its Applications. Cambridge
  University Press, 1994.

\bibitem[Rud91]{rudin1991functional}
Walter Rudin.
\newblock {\em Functional Analysis}.
\newblock International series in pure and applied mathematics. McGraw-Hill,
  1991.

\bibitem[SFL11]{Fukumizu2011rkhs}
Bharath~K. Sriperumbudur, Kenji Fukumizu, and Gert~R.G. Lanckriet.
\newblock {Universality, Characteristic Kernels and RKHS Embedding of
  Measures}.
\newblock {\em Journal of Machine Learning Research}, 12(70):2389--2410, 2011.

\bibitem[SHS01]{Schoelkopf2001representer}
Bernhard Sch\"olkopf, Ralf Herbrich, and Alexander~J.\ Smola.
\newblock A generalized representer theorem.
\newblock {\em Computational Learning Theory}, pages 416--426, 2001.

\bibitem[Sto77]{Stone1977}
Charles~J. Stone.
\newblock Consistent nonparametric regression.
\newblock {\em The Annals of Statistics}, 5(4):595--620, 1977.

\bibitem[Sto80]{Stone1980}
Charles~J. Stone.
\newblock Optimal rates of convergence for nonparametric estimators.
\newblock {\em The Annals of Statistics}, 8(6):1348--1360, 1980.

\end{thebibliography}
